\documentclass[11pt,reqno]{amsart}
\usepackage{geometry}                
\usepackage{graphicx}
\usepackage{amssymb}\usepackage{amsthm}
\usepackage{epstopdf}
\usepackage{subfig}
\usepackage{sidecap}
\usepackage{color}
\theoremstyle{plain}
\newtheorem{thm}{Theorem}[section]

\newtheorem{cor}[thm]{Corollary}
\newtheorem{prop}[thm]{Proposition}
\newtheorem{rem}[thm]{Remark}
\newtheorem{ex}[thm]{Example}

\title{THE INVERSE-DEFORMATION APPROACH TO FRACTURE}
\author{Phoebus Rosakis$^{1,2}$, Timothy J. Healey$^{3,4}$ \& U\u{g}ur Alyanak$^3$}

\address[1]{Department of Mathematics and Applied Mathematics, University of Crete, Heraklion 70013 Crete, Greece. Email: rosakis@uoc.gr}
\address[2]{Institute of Applied and Computational Mathematics, Foundation for Research and Technology-Hellas, Voutes 70013 Crete, Greece}
\address[3]{Department of Mathematics, Cornell University, Ithaca, NY 14853, USA. Email: tjh10@cornell.edu}
\address[4]{Field of Theoretical and Applied Mechanics, Cornell University, Ithaca, NY 14853, USA}

\begin{document}
\maketitle
\begin{abstract}
 We propose  a one-dimensional,  nonconvex elastic constitutive model  with higher gradients that can predict spontaneous fracture at a critical load via a  bifurcation analysis. It overcomes the problem of discontinuous deformations without additional  field variables, such as damage or phase-field variables, and without \textit{a priori} specified surface energy. Our main tool  is the use  of the inverse deformation, which can be extended to be a piecewise smooth mapping even when the original deformation has discontinuities describing cracks opening.  We exploit this via the inverse-deformation formulation of finite elasticity due to Shield and Carlson, including higher gradients in the energy. The problem is amenable to a rigorous global bifurcation analysis in the presence of a unilateral constraint. Fracture under hard loading occurs on a bifurcating solution branch at a critical applied stretch level and fractured solutions are found to have surface energy arising from higher gradient effects.\end{abstract}

\numberwithin{equation}{section}  
\newcommand{\sw}{{W^\ast}}
\newcommand{\ssw}{{\buildrel \ast\ast\over W}}
\newcommand{\beq}{\begin{equation}}
\newcommand{\eeq}{\end{equation}}
\newcommand{\la}{{\lambda}}
\newcommand{\si}{{\sigma}}
\newcommand{\TT}{\mathcal{T}}
\newcommand{\ep}{{\varepsilon}}
\newcommand{\bpr }{\begin{proof}}
\newcommand{\epr}{\end{proof}}
\newcommand{\tred}{\textcolor{black}}
\newcommand{\vp}{{\varpi}}

\maketitle

\section{Introduction}

\noindent A major difficulty in modelling brittle fracture of solids is that  cracks are usually represented via discontinuous deformations. Whereas discontinuous gradients of the deformation can be described by nonlinear elasticity, and can be regularized by higher gradient terms in the stored energy, it is not clear how to deal with discontinuities of the deformation in a similar way. To circumvent this issue, various ingenious  fracture models have been developed that introduce the crack as a separate entity, \textit{a priori} endowed with properties that are distinct from the constitutive law of the bulk material, such as surface energy, cohesive laws, damage variables or phase fields, e.g., \cite{griffith,baren,ambr,franc,bourdin,triant}. 

A distinctive approach is Truskinovsky's treatment of fracture as a phase transition \cite{trusk}, where fracture results from nonconvexity of a two-well stored energy function, with the  transformation strain (location of the second well) going to infinity; it involves strains in the fracture zone that become unbounded  in that limit. It would be desirable to regularize this nonconvex problem using higher gradients, but this leads to unbounded energies.  So, while we embrace the idea of  fracture as a phase change, the question remains how to model fracture in this spirit, somehow avoiding singularities and additional constitutive ingredients for the crack.
\begin{figure}
  \centering
 \subfloat[]{\includegraphics[width=0.42\textwidth]{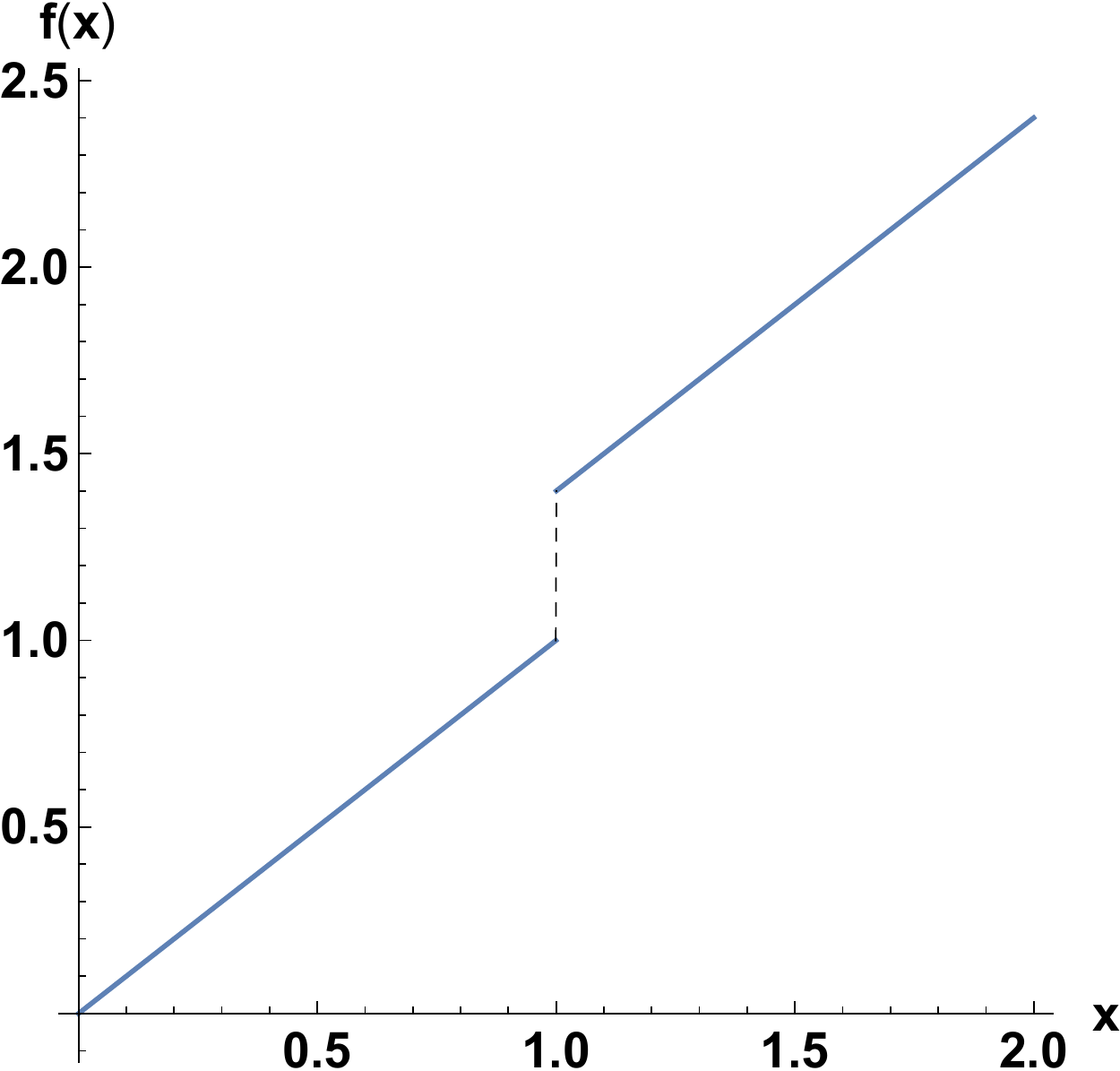}
  \label{Fig1a}}
  \hspace{0.8cm}
 \subfloat[]{\includegraphics[width=0.42\textwidth]{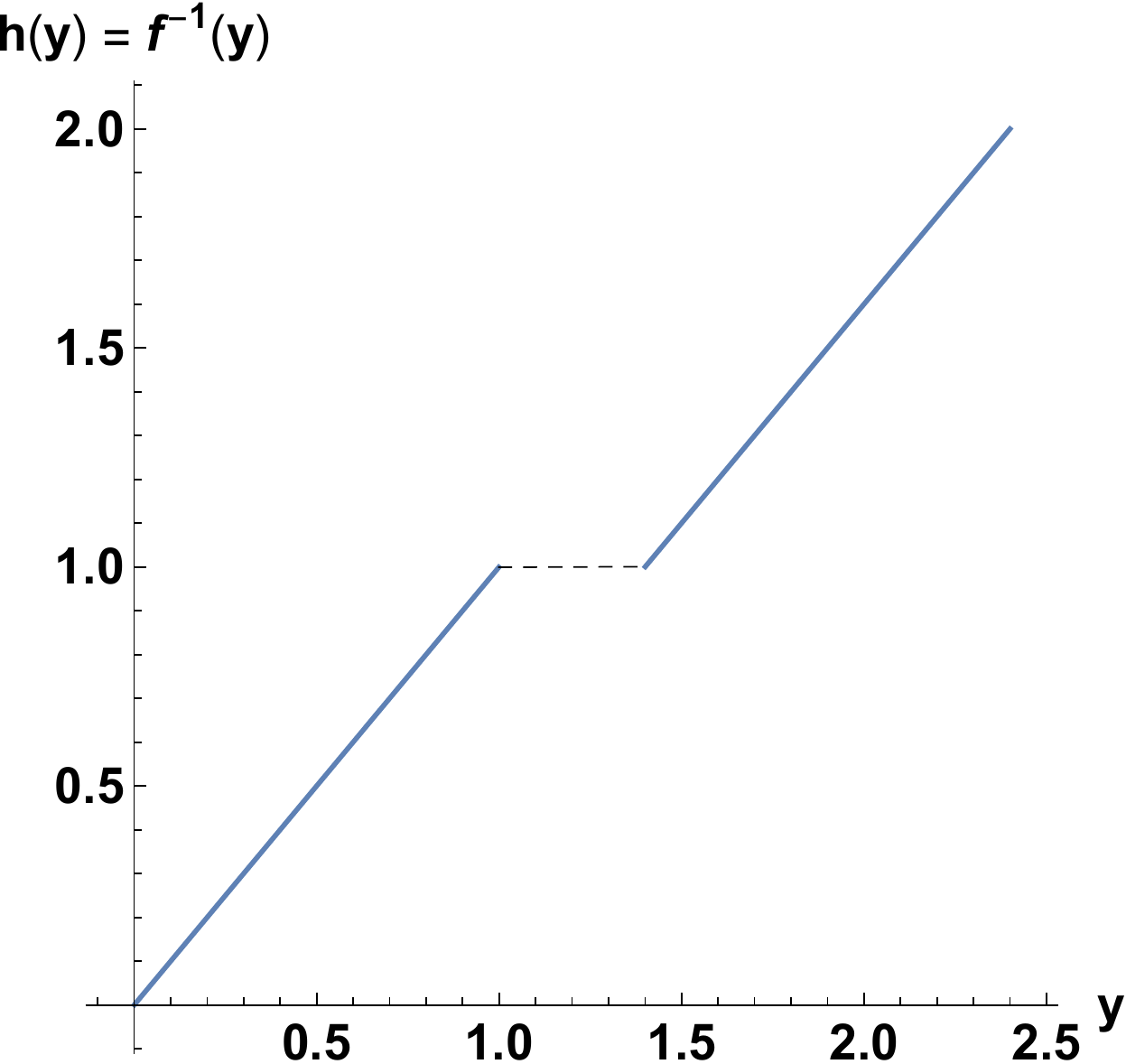}
  \label{Fig1b}} 
  \caption{(A) A cracked deformation (blue) with a vertical segment (black) attached to render the graph a continuous curve. (B)  The generalized inverse of the cracked deformation. The horizontal segment (black) is the opened crack in the deformed configuration.}
  \label{Fig1}
\end{figure}

Here we introduce a local, elastic, nonconvex constitutive model that  is amenable to regularization  by higher gradients. It can predict  fracture with the spontaneous appearance of discontinuous deformations via  bifurcation of equilibria, but does not involve additional field variables, such as damage or phase-field variables, or an \textit{a priori} specified surface energy for cracks. Our main tool is the inverse deformation, which can be extended to a be piecewise smooth mapping even when the original deformation has discontinuities describing crack opening.  We exploit this through the inverse-deformation formulation of finite elasticity due to Shield  \cite{shield}. 

\subsection*{Motivation} Our first observation concerns one-dimensional fractured deformations, viewed here as strictly monotone mappings  that involve at least one jump discontinuity (a crack). The graph of such a function has disjoint pieces that can be joined together by vertical segments of ``infinite slope''  as in  Fig.~\ref{Fig1a}. One can then easily construct a generalized inverse of this deformation by interchanging the abscissa and the ordinate. The graph of this inverse has strictly increasing pieces that are the graphs of the (standard) inverses of the deformation on either side of the crack. Moreover their graphs are connected by a horizontal segment, which corresponds to the ``inverse'' of the segment of infinite slope. The generalized inverse so constructed is a piecewise smooth mapping, where two or more ``phases'' of positive stretch are separated by one or more ``phases'' of zero stretch. To allow for the latter,  we extend the notion of deformation to admit nonnegative, as opposed to strictly positive, derivatives, so that mere non-strict monotonicity is required.

In a sense, the inverse deformation closes the crack, as it maps each crack interval to a single point, the reference location of the crack. In Fig.~\ref{Fig1}, the interval $1\le y\le 1.4$ is mapped back to the point $x=1$. Analogously, the original deformation opens the crack (at $x=1$ in the example), as it maps the single crack point in the reference configuration to the cracked interval  in the deformed configuration ($1\le y\le 1.4$ in the example, or the gap between the two opened-crack faces). A major advantage is that unlike the discontinuous original deformation,  the generalized inverse, Fig.~\ref{Fig1b}, is Lipschitz continuous   and  has gradient discontinuities,  like a  two-phase deformation \cite{ericksen}.  Here, intervals of positive inverse stretch are separated by intervals of zero inverse stretch. These we identify with the uncracked phase and  the cracked phase, respectively.  The length of a cracked-phase interval is nothing but  the crack opening displacement.

 \begin{figure}
  \centering
 \subfloat[]{\includegraphics[width=0.42\textwidth]{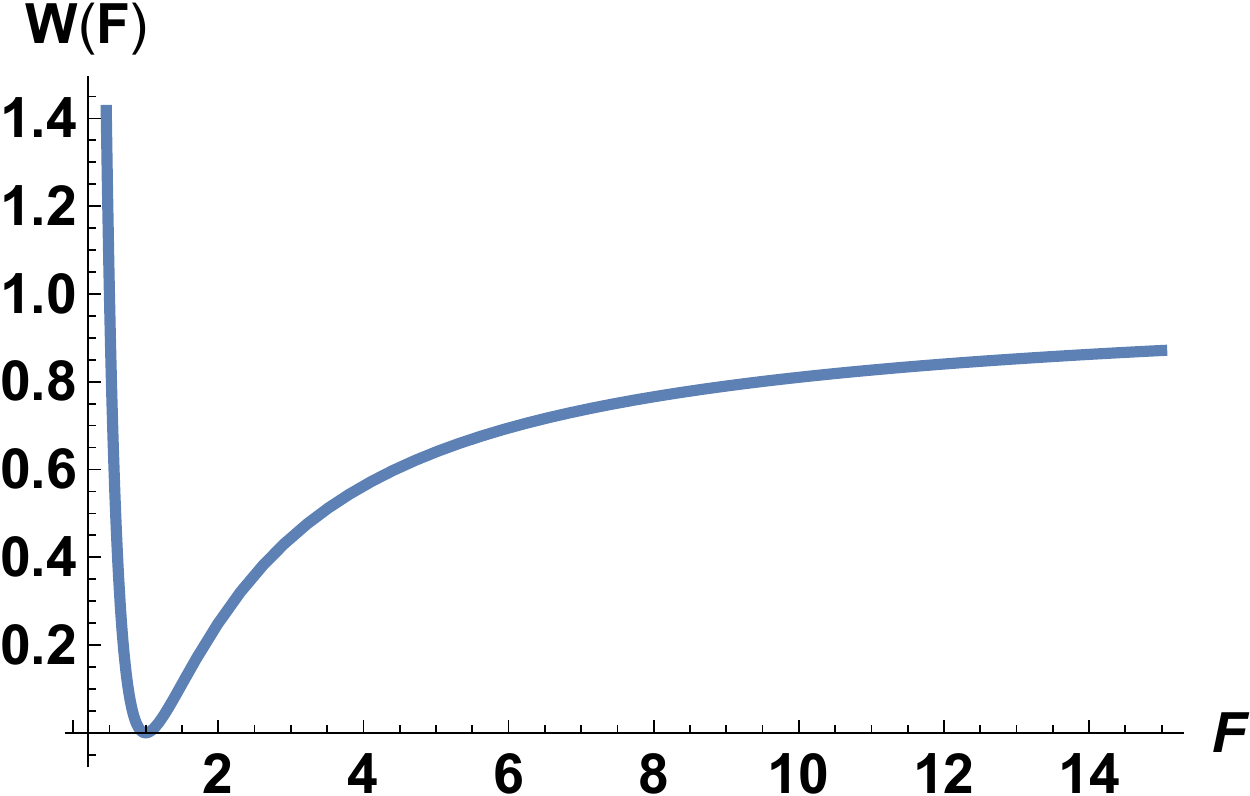}
  \label{Fig2a}}
  \hspace{0.8cm}
 \subfloat[]{\includegraphics[width=0.42\textwidth]{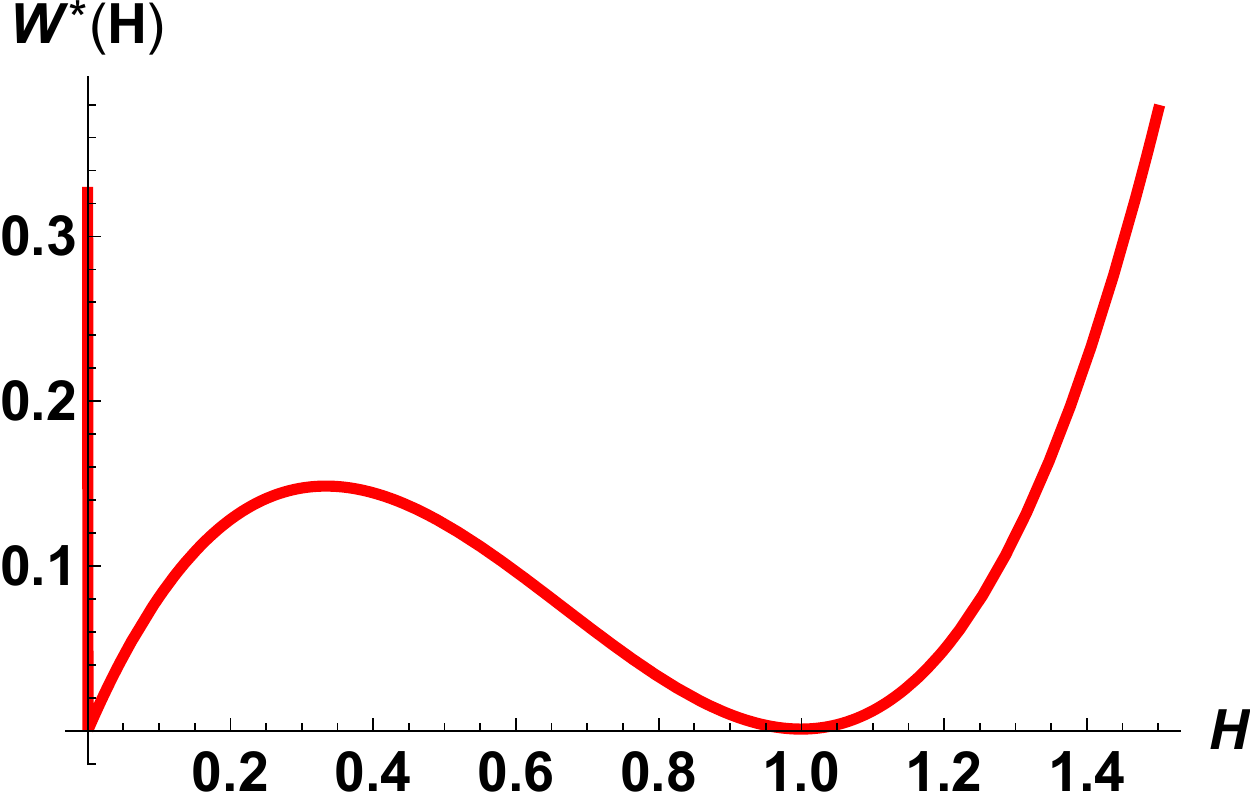}
  \label{Fig2b}} 
\caption{(A) A stored energy function of a material that undergoes brittle fracture  (cf. \eqref{WW} in Example \ref{ex5.10}). (B) The corresponding inverse stored energy function $\sw$.}
  \label{Fig2}
\end{figure}

Another crucial advantage of the inverse description is that the analogy with phase transitions 
extends naturally to the constitutive law itself, once we invoke the inverse deformation approach of Shield \cite{shield}, as we now explain. A material suffering brittle fracture in a one-dimensional setting is typically characterized by an elastic stored energy function of the form shown in Fig.~\ref{Fig2b}, having a convex well at the reference state, but eventually becoming concave and
approaching a horizontal asymptote from below as the stretch tends to infinity. The inverse
stored energy function  $ \sw $  is related to the (usual) stored energy function  $W$  by \cite{shield} 
$$\sw(H) = HW(1/H),\quad H>0,$$
and has the property that the elastic energy of a deformation  $ f : [0,1] \to [0,\la] $  can  be
written as
$$\int_0^1 W (f'(x))dx = \int_0^\la \sw (h'(y))dy,
$$
where  $ h = f^{-1} $  is the inverse deformation, and the \textit{inverse stretch} is 
$$H(y)=h'(y)=1/f'(x), \quad \hbox{where }x=f^{-1}(y).$$ 
For W as in Fig.~\ref{Fig2a}, the inverse stored energy  $ \sw $ 
would be as in Fig.~\ref{Fig2b},, where we now extend its domain of definition as follows
\beq\label{shield}  \sw (H)=\begin{cases}
HW(1/H),& H>0\\
0,& H=0\\
\infty,& H<0.\\
\end{cases}\eeq
Here we allow the possibility that the inverse stretch  $ H=h'= 0 $  (corresponding to the cracked phase as discussed above), but prohibit  $ h'< 0 $, which corresponds to  $ f'<0 $, namely,
orientation-reversing interpenetration. We recall that the case  $ h'= 0 $  corresponds to crack opening, not interpenetration. Instead of the usual
constraint  $ f' > 0 $,  we thus impose  $ h'\ge0 $  as a  constraint. When one visualizes this unilateral constraint as a vertical barrier just to the left of $ 0$ as part of the graph of  $ \sw $  (Fig.~\ref{Fig2b}), it is clear that the latter  has the form of
a two-well energy, with wells at $H=0$ and $H=1$. The inverse deformation of Fig.~\ref{Fig1b} is  a zero-energy one (global minimizer) provided that the rising portions  have slope $H=1$, since the flat portion has slope $H=0$.  In Fig.~\ref{Fig2b}, the well at $H=1$ corresponds to the undeformed, uncracked ``phase'', while the well at $H=0$ to the cracked ``phase''. \tred{The length of the interval in the  $H=0$  phase (horizontal segment  in Fig.~\ref{Fig2b}) is the crack opening displacement. In this sense, the  $ 0 $  phase is ``thin air'' or empty space between crack faces. To see this, consider local mass balance in terms of the inverse deformation, which reads  
\beq\label{mb}\rho_0 H=\rho,\eeq
where $\rho_0>0$ is the reference density and $\rho$ the deformed density. Clearly, for $H=0$ this implies $\rho=0$, absence of matter, or empty space.  Our viewpoint is that fracture corresponds to  two-phase inverse deformations minimizing this two-well energy, where intervals in the $H=0$ phase are opened cracks in the deformed configuration.}

A standard relaxation of this two-well  energy gives the convexification  $ \ssw $  of  $ \sw $  (together with the constraint  $ h'\ge0 $ ). This corresponds to a material that cannot resist compression (Fig.~\ref{Fig3a}). Also  $ \ssw $  is the inverse stored energy function of  $ \widehat W $, which has the form shown in Fig.~\ref{Fig3b}, corresponding to a material that cannot sustain tension. The drawback here is that minimizers  $ h $  of the inverse energy can have an arbitrary number of cracked intervals  $ H=0 $  alternating with intervals where  $ H=1 $,  corresponding to arbitrary positions of cracks in the reference configuration. Another problem is that this material breaks at the slightest pull.
\begin{figure}
  \centering
 \subfloat[]{\includegraphics[width=0.42\textwidth]{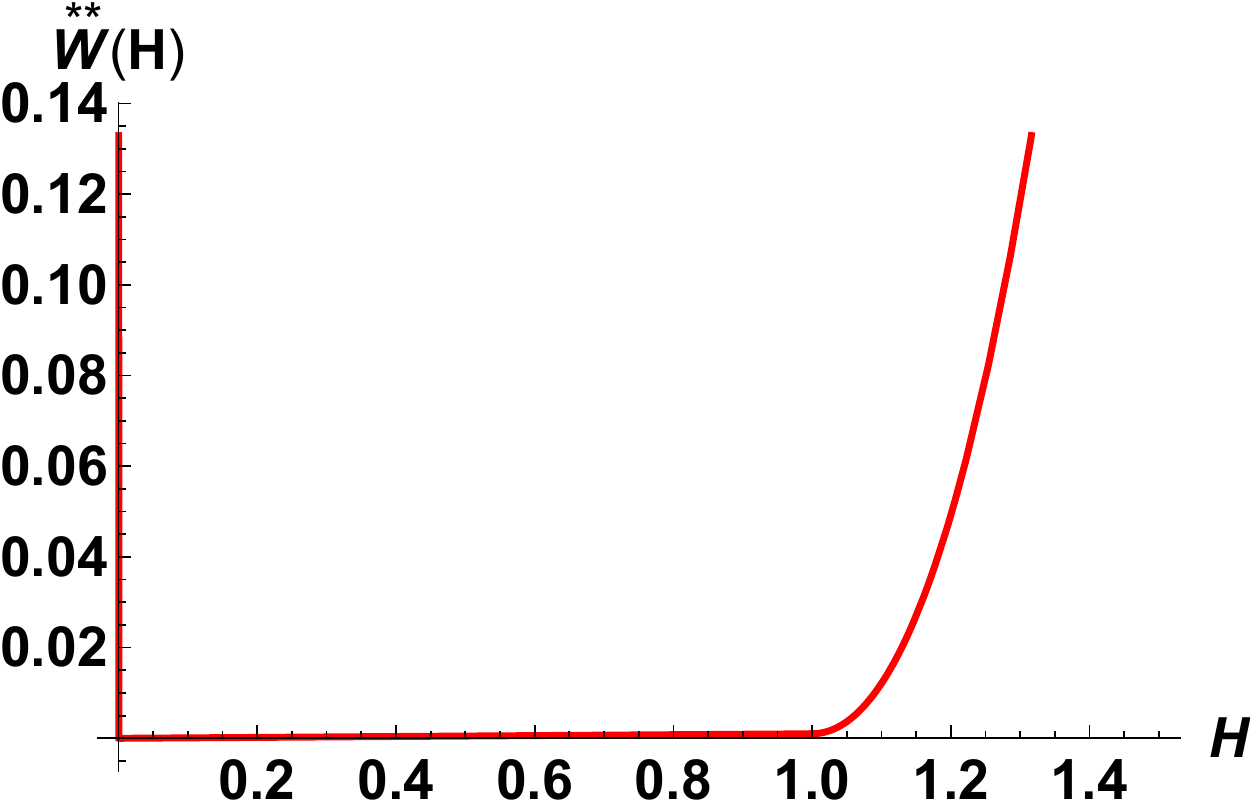}
  \label{Fig3a}}
  \hspace{0.8cm}
 \subfloat[]{\includegraphics[width=0.42\textwidth]{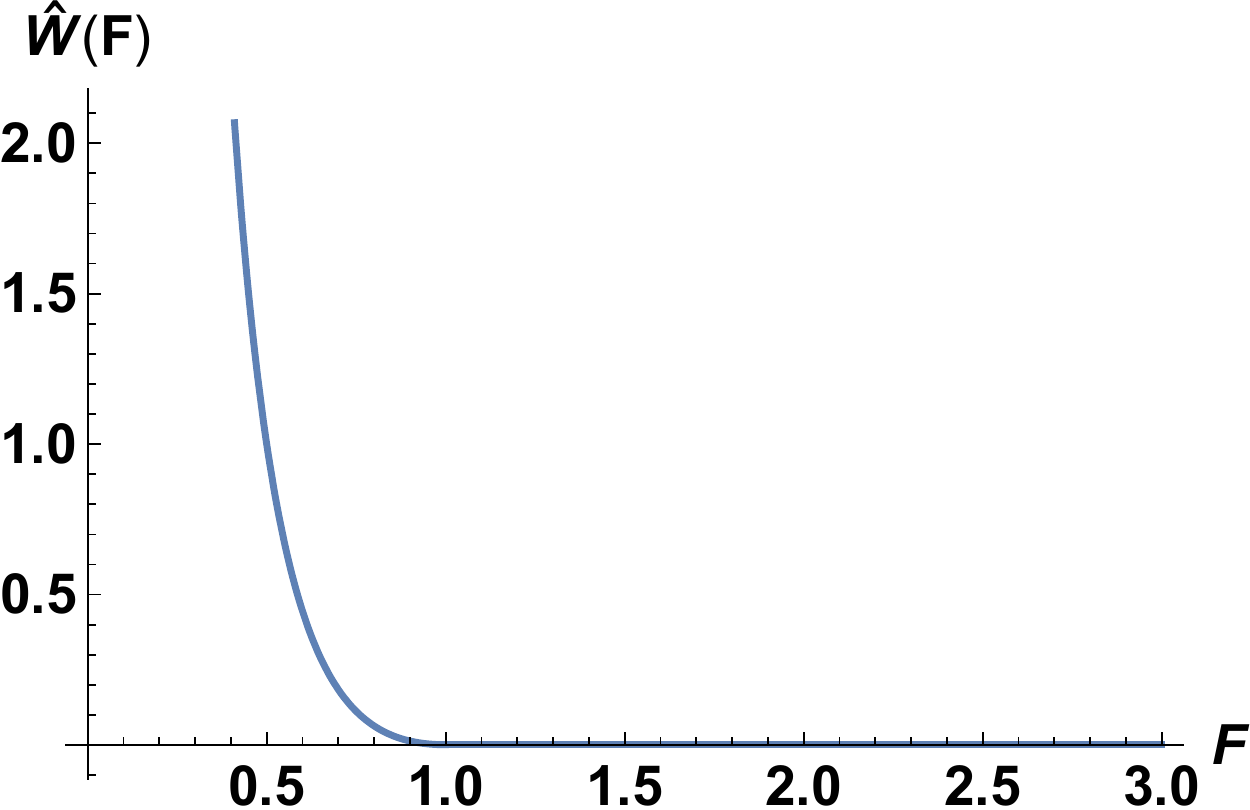}
  \label{Fig3b}}
  \caption{(A) The convexification   $ \ssw $  of  $ \sw $  of Fig.~\ref{Fig2b}. (B) The  stored energy function  $ \hat W $  whose inverse stored energy is  $ \ssw $.}
  \label{Fig3}
\end{figure}

To fix these problems, we exploit the final advantage of the inverse approach. The problem associated with  $ \sw $  can be regularized by the addition of higher gradients of the inverse deformation  $ h $  to the energy which would become
\beq\label{ehg} E_\ep\{h\}= \int_0^\la \sw (h'(y))dy
+\frac{\ep}{2}\int_0^\la  [h''(y)]^2 dy,\eeq
subject to the unilateral constraint  $h'\ge 0$ on $[0,\la]$. The analogous attempt to add  higher gradients of the original deformation  $ f $  to the energy  runs into difficulties because of the discontinuities of  $ f $  (Fig.~\ref{Fig1a}); such  deformations cannot be approximated by smooth functions and still maintain bounded  energy,  if higher-gradients are included  in the usual way.   \tred{A major strength of the inverse approach lies in the simplicity of the model energy \eqref{ehg} and the relative mathematical ease with which its equilibria are studied. Even though the original deformation is discontinuous, the inverse deformation can be extended to be Lipschitz.  In fact, we show that equilibria of the second-gradient energy \eqref{ehg} are $C^1$,  including intervals of zero inverse stretch $h'=0$,  corresponding to opened cracks in the deformed configuration. In the presence of such intervals, the original deformation is discontinuous.}

 \subsection*{Methods.}
\tred{The inverse formulation of Shield \& Carlson \cite{shield,carlson}, combined with inspiration from  Truskinovsky's  idea of fracture  as a phase transition \cite{trusk} ,  allows us to treat the problem as a constrained, but otherwise standard, two-well elasticity problem with higher gradients. We note that the model does not involve any special treatment for cracks, such as separate cohesive energies, different spatial scales for crack zones,  ``exotic'' spaces such as SBV, or additional phase fields.}This general approach is promising, and it begs the question of two or three dimensional formulations, which we pursue elsewhere \cite{rosakis}.  

 We study equilibria of the displacement problem in the inverse formulation, taking $E_\ep\{h\}$ in  \eqref{ehg}, to be the energy in terms of the inverse deformation $h$.  We employ techniques of global bifurcation theory \cite{rabin}, keeping in mind that stable branches of local energy minima may occur, while exploiting phase plane techniques in the spirit of \cite{cgs}. The only complication here is the unilateral constraint  $ h'\ge0 $  on the inverse deformation,  in an otherwise fairly standard  two-well problem with higher gradients like  \cite{cgs,lifsh,triant}. We formulate the problem as a variational inequality incorporating the constraint, and  employ the methods of  \cite{globi}.
 
\tred{To obtain quantitative information on bifurcating solution branches, we choose specific examples of the stored-energy function $W$ of the form shown in Fig.~\ref{Fig1b}, and either compute solutions using the bifurcation/continuation program AUTO \cite{auto}, or obtain branches analytically or semi-analytically in some cases.}

 \begin{figure}
  \centering
\includegraphics[width=0.8\textwidth]{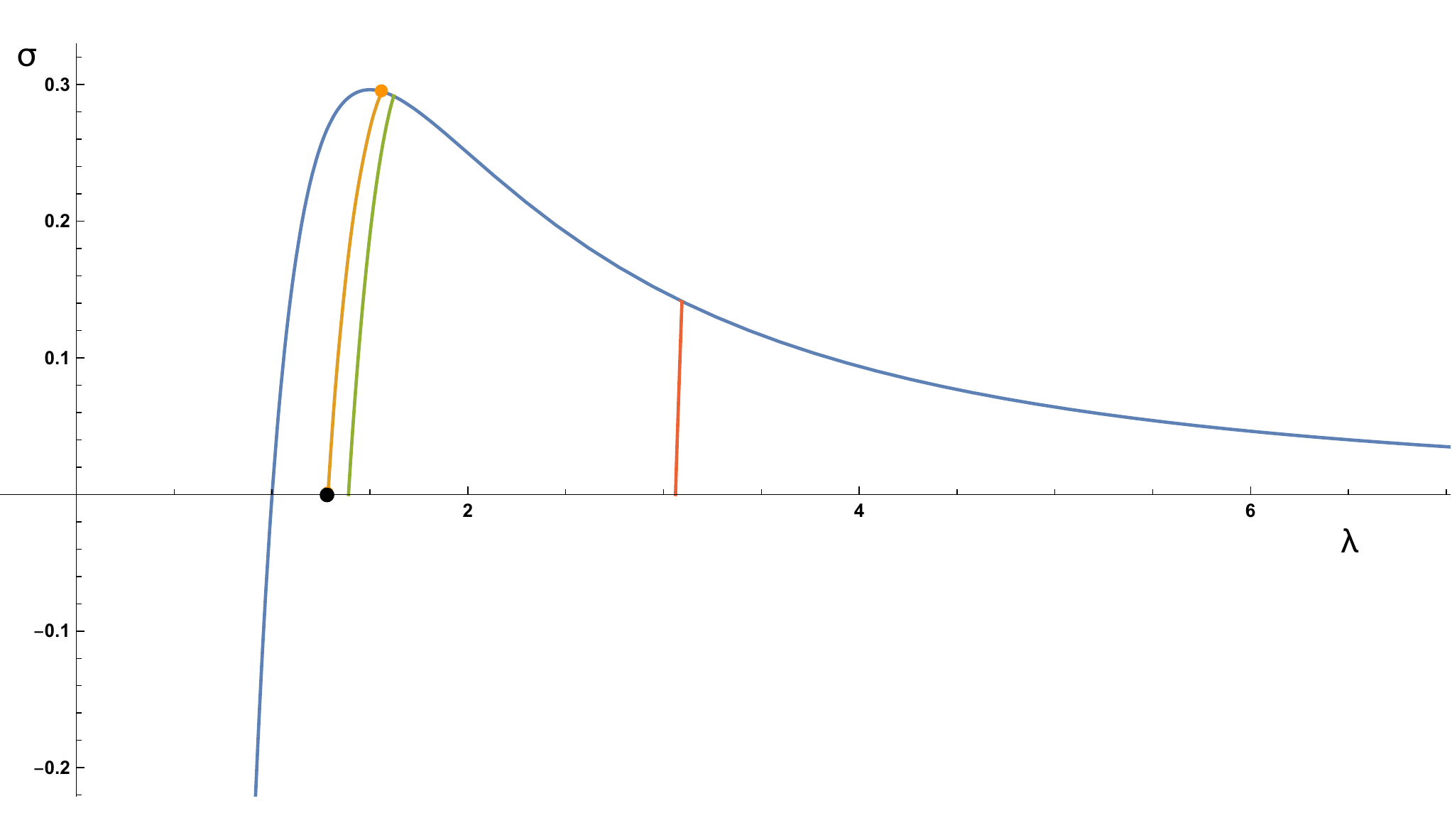}
  \caption{The trivial solution branch and first bifurcating branch for various values of  $ \ep $, represented in the $(\si,\la)$ or stress  \textit{vs} average-stretch plane  (for the special constitutive law \eqref{WW} in Example \ref{ex5.10}). Blue curve: trivial solution branch corresponding to homogeneous deformation. Orange, green and red curves, first bifurcating branch for different values of  $ \ep=2/49 $,  $ 2/25 $  and  $ 2 $, respectively. For example, if  $ \ep=2/49 $, the trivial branch is stable to the left of the orange dot (1st bifurcation point) and  unstable to the right.  The horizontal axis to the right of the black dot (lowest point of the orange curve) consists of stable broken solutions. Similarly for the other cases.}
  \label{Fig4}
\end{figure} 

 \subsection*{Results.}
 The results are consistent with our expectations of one-dimensional brittle fracture, and agree with some predictions of discrete models \cite{bdg}:  In particular, pulling on a bar with prescribed end displacement deforms it homogeneously until the end displacement reaches a critical level.  After that, the stress drops suddenly to zero and remains there during further elongation; the bar is broken See Fig.~\ref{Fig4}. On the end-load  \textit{vs}  average-stretch diagram this corresponds to the portion of the blue curve to the left of the orange dot (for  $ \ep=2/49 $). The blue curve corresponds to the trivial branch of solutions with uniform stretch. The orange dot is the first bifurcation point, just after the maximum of the stress-stretch curve for small  $ \ep $. Beyond that the homogeneous stretch solution (blue curve) is unstable. The first nontrivial branch of bifurcating solutions (orange curve) eventually connects to the zero-stress  axis at the black dot in Fig.~\ref{Fig4}. This point corresponds to initiation of fracture. 
 
 The first branch then continues along the horizontal zero-stress axis to the right of the black dot.  The bar breaks at one of the two ends, and the stress vanishes thereafter, as the end displacement is further increased. Points on the horizontal axis to the right of the black dot in Fig.~\ref{Fig4} correspond to broken solutions. The broken bar has vanishing stress and is virtually undeformed except for a transition layer, followed by a closed interval of zero inverse stretch, which corresponds to the opened crack. The length  of this interval equals the crack opening displacement. 
 
 Broken solutions are  endowed with  additional surface energy, because of higher gradients.  To leading order in the higher gradient coefficient  $ \ep $, this surface energy is determined explicitly given the stored energy function:
$$\hbox{Surface Energy}=\sqrt{\ep}\int_0^1 \sqrt{2W^\ast (H)} dH  = \sqrt{\ep}\int_1^\infty \sqrt{2W (F)/F^5} dF. $$
It  plays a role similar to the one posited by Griffith. 

The inverse deformation is smooth, but the original deformation is discontinuous at the broken end. All solutions with more  than one fracture are unstable; they arise from  higher-mode branches bifurcating off the homogeneous solution, that are all unstable.  

Longer bars are more brittle (break sooner and more suddenly) than shorter bars. This is  because rescaling shows that  the energy \eqref{ehg} of a bar of reference length  $ L $  and higher gradient coefficient  $ \ep $  is equal  to the energy of a bar of unit reference length and a higher gradient coefficient  $ \bar\ep=\ep/L^2 $, while the fracture stretch increases with  $ \ep$.  In particular, very short bars, or nanoparticles,  are stable in uniform stretch well beyond the maximum of the stress-stretch curve, as observed by Gao \cite{gao}.  
\par
\tred{Some of our predictions differ from the results of some other models.
 First, the crack faces are sharply defined points, and the crack is empty of matter, in contrast to damage or phase field models, e.g.,  \cite{bourdin}, where the crack is diffuse. This occurs despite the smoothening effect of higher gradients, and is related to the unilateral constraint  $ h'\ge0 $.
 Second, upon fracture, the stress suddenly drops suddenly to zero at a finite macroscopic stretch, and stays zero thereafter, instead of approaching zero for large elongation, which occurs in some cohesive-zone and nonlocal models, e.g., \cite{dPTrusk,triant}. This agrees with the common concept of brittleness. We note that the  stress-stretch constitutive relation underlying our model only asymptotically approaches zero as the stretch goes to infinity. The sudden drop to zero stress at finite applied stretch is a consequence of material instability and bifurcation to an inhomogeneous state.
Finally, our formulation and results are relevant to large deformations, the only setting in which the inverse-deformation approach makes sense.}

We now give an outline of the work. In Section 2 we consider the 
inverse-stretch formulation, presuming hard loading in the presence of both 
an integral constraint (for compatibility with imposed end displacements) and the unilateral constraint of nonnegative inverse stretch. 
We demonstrate the existence of a global energy minimizer, and 
 proceed to formulate the Euler-Lagrange variational inequality governing 
all equilibria. We reformulate the variational inequality accounting for the 
integral constraint and demonstrate that all equilibria are  $C^{1} $.  We end 
the section with some \textit{a priori} bounds. In Section 3 we prove the existence 
of global solution branches bifurcating from the trivial, homogeneous 
solution. We employ methods of global bifurcation for variational inequalities  \cite{globi} to obtain branches of nontrivial (non-homogeneous) solutions via the 
methodology of  \cite{rabin}, and we establish nodal properties of solutions as in 
 \cite{cran}. The latter, combined with the \textit{a priori} bounds, imply that all bifurcating 
branches of equilibria are unbounded. In Section 4 we obtain more 
detailed properties of solutions via qualitative phase-plane arguments. In 
particular, we show that each global solution branch has a bounded component
\tred{that connects the bifurcation point to another solution exhibiting the 
onset of fracture, marked by vanishing of the inverse stretch at some point.} The \textit{a priori} bounds of Section 2 play a key role here. 
Our qualitative analysis also reveals the structure of all further broken 
solutions on the complementary, unbounded component of the solution branch. \tred{These broken solutions are characterized by the presence of nonempty, closed intervals of zero inverse stretch}. We finish 
the section with some stability results. We demonstrate that the 
trivial solution is locally stable up to the first bifurcation point, and unstable beyond it, and we show that 
all \tred{bifurcating solutions of mode higher than 1} are also unstable. In Section 5 we convert our 
results, \textit{a posteriori}, to the original Lagrangian variables. In particular, 
we obtain the effective end-load  \textit{vs}  average stretch curve as in  \cite{lifsh} via projection of the first 
global solution branch. Here the homogeneous solution appears as the nominal 
constitutive law (stress-stretch relation). For large enough average stretch, we observe that the only possibility for an energy minimizer (under hard loading) is the broken solution along the first bifurcating branch.  We finish the section \tred{with a semi-analytical characterization of the first branch and the fracture stretch}, and some concrete results for specific models---some analytical and some numerical.

\section{Formulation and  A Priori Estimates}

We start with assumptions for  $ W $  that are typical of one-dimensional brittle fracture  and in accordance with the 
properties suggested in Fig.~\ref{Fig2}:
\begin{equation}
\begin{array}{l}
 W\in C^{3}(0,\infty );\mbox{\, }W(F)\nearrow \infty \mbox{\, as\, 
}F\searrow 0;\mbox{\, }W(F)\nearrow \gamma >0\mbox{\, as\, }F\nearrow \infty 
; \\ 
 W(1)=0;\mbox{\, }W(F)>0,\mbox{\, }F\ne 1; \\ 
 W\mbox{\, is\, strictly\, convex\, on\, }[0,1/\kappa )\mbox{\, and\, 
strictly\, concave\, on\, }(1/\kappa ,\infty ),\mbox{\, where\, \, }0<\kappa 
<1. \\ 
 \end{array}\label{eq-1}
\end{equation}
As a consequence of \eqref{shield}, and letting  $ \dot\sw(H)=d\sw(H)/dH $, it follows that
\begin{equation}
\begin{array}{l}
 W^{\ast }\in C^{3}(0,\infty );\mbox{\, }W^{\ast }(0)=W^{\ast 
}(1)=0;\mbox{\, }W^{\ast }(H)>0,\mbox{\, }H\ne 0,1;\mbox{\, }\Dot{{W}}^{\ast 
}(0^{+})=\gamma ; \\ 
 W^{\ast }\mbox{\, is\, strictly\, concave\, on\, }[0,\kappa )\mbox{\, and\, 
strictly\, convex\, on\, }(\kappa ,\infty ). \\ 
 \end{array}\label{eq-2}
\end{equation}
We write  the total potential energy \eqref{ehg}  in terms of the 
inverse stretch  $ H={h}' $  (derivative of the inverse deformation  $ h $):
\begin{equation}
\label{eq1}
E_{\varepsilon } [H]=\int_0^\lambda {[\frac{\varepsilon 
}{2}({H}')^{2}+W^{\ast }(H)]dy} ,
\end{equation}
with  $ \varepsilon ,\lambda >0$. 
We consider ``hard'' loading, namely Dirichlet boundary conditions on the original deformation  $ f(0)=0,\mbox{\, }f(1)=\lambda $. Here  $ \lambda>0 $  is the prescribed deformed length of the bar,  whose reference length equals unity.  As a result,  the inverse deformation is subject to   $ h(0)=0,\mbox{\, }h(\lambda )=1 $.  The inverse stretch is thus subject to two constraints
\begin{equation}
\label{eq2}
 \int_0^\lambda {H(y)dy} =1, \qquad
 H \ge  0\mbox{\, on\, [}0,\lambda ],  \end{equation}
for  $ \lambda \in (0,\infty ) $. The second above  allows for  $ H=0 $,  which, as discussed in the Introduction, corresponds to crack opening. For convenience, we 
change variables as follows: 
\begin{equation}
\label{eq3}
y=\lambda s,\mbox{\, }0 \le  s \le  1; \qquad
 u(s):=\lambda H(\lambda s)-1.\end{equation}
Then \eqref{eq1}, \eqref{eq2} are equivalent to 
\begin{equation}
\label{eq4}
V_{\varepsilon } [\lambda ,u]=\int_0^1 {[\frac{\varepsilon 
}{2}({u}')^{2}+\lambda^{4}W^{\ast }([1+u]/\lambda )]ds} ,
\end{equation}
subject to
\begin{equation}
\int_0^1 {u(s) ds} =0; \qquad 
 u \ge  -1\mbox{\, on\, [}0,1], 
\label{eq5} \end{equation}
respectively. 

We define the Hilbert space
\begin{equation}
\label{eq6}
\mathcal{H}:=\{v\in H^{1}(0,1):\int_0^1 {vds=0\},} 
\end{equation}
with inner product
\begin{equation}
\label{eq7}
\left\langle {u,v} \right\rangle :=\int_0^1 {{u}'{v}'ds} ,
\end{equation}
along with the admissible set
\begin{equation}
\label{eq8}
\mathcal{K}:=\{v\in \mathcal{H}:v \ge  -1\mbox{\, on\, }[0,1]\},
\end{equation}
which is closed and convex. We remark that the inner product \eqref{eq7} 
on  $ \mathcal{H} $  is equivalent to the usual  $ H^{1}(0,1) $  inner product, by 
virtue of a Poincare inequality.

Since the integrand in \eqref{eq4} is quadratic and convex in the argument  $ {u}'$,  
the following is standard:
\begin{prop}\label{prop2.1} $ u\mapsto V_{\varepsilon } [\lambda ,u] $  { attains its minimum on } $ \mathcal{K} $  {for each}   $ \lambda \in (0,\infty )$. 
\end{prop}
\begin{proof}Clearly  $ V_{\varepsilon } [\lambda ,\cdot ] $  is coercive and 
sequentially weakly lower semi-continuous on $ \mathcal{H}$.  Let  $ \{u_{k} 
\}\subset \mathcal{K} $  be a minimizing sequence. Then for a subsequence, 
 $ u_{k_{j} } \rightharpoonup u $  weakly in  $ \mathcal{H}$,  and by compact 
embedding,  $ u_{k_{j} } \to u $  uniformly on  $ [0,1]$.  Since each  $ u_{k_{j} } 
 \ge  -1 $  on  $ [0,1]$, it follows that  $ u\in \mathcal{K}$.  
\end{proof}
Now let  $ u,v\in \mathcal{K}$,  and consider  $ e(t):=E[(1-t)u+tv],\mbox{\, }t\in 
[0,1] $. Then  $ u $  is a critical point (equilibrium) if  $ {e}'(0) \ge  0$,  
which yields the (Euler-Lagrange) variational inequality
\begin{equation}
\label{eq9}
\int_0^1 {[\varepsilon {u}'(v-u{)}'+\lambda^{3}\Dot{{W}}^{\ast }\left( 
{[1+u]/\lambda } \right)(v-u)]dx}  \ge  0\mbox{\, for\, all\, }v\in 
\mathcal{K},
\end{equation}
where  $ \Dot{{W}}^{\ast }(H):=\frac{d}{dH}W(H)$.  We establish some general 
properties of solutions of \eqref{eq9}, postponing for now the construction of 
other equilibria beyond global energy minimizers. Define the interior 
\begin{equation}
\label{eq10}
\mathcal{K}^{o}:=\{u\in \mathcal{K}:u>-1\mbox{\, on\, [0,1]\},}
\end{equation}
and suppose that  $ u\in \mathcal{K}^{o} $  satisfies \eqref{eq9}. Let  $ \psi ,\varphi 
\in H^{1}(0,1) $  such that $ \int_0^1 {\varphi dx=1,}  $ and set  $ h=\psi -\varphi 
\int_0^1 {\psi dx.}  $  Then  $ h\in \mathcal{H}$,  and  $ v=u\pm th\in 
\mathcal{K}$,  for  $ t>0 $  sufficiently small. Employing these two test 
functions in \eqref{eq9}, we conclude that
\begin{equation}
\label{eq11}
\int_0^1 {[\varepsilon {u}'{\psi }'+\{\lambda^{3}\Dot{{W}}^{\ast }\left( 
{[1+u]/\lambda } \right)-\mu \}\psi ]ds} =0\mbox{\, for\, all\, }\psi \in 
H^{1}(0,1),
\end{equation}
where  $ \mu  $  is a constant (Lagrange) multiplier. This leads to:
\begin{prop}\label{prop2.2}\textit{An interior point} $ u\in \mathcal{K}^{o} $ \textit{is a solution of }\eqref{eq9} for  $ \lambda \in 
(0,\infty )$, if and only if $ u\in C^{2}[0,1] $ \textit{ satisfies}
\begin{equation}
\label{eq12}
\begin{array}{l}
 -\varepsilon {u}''+\lambda^{3}\Dot{{W}}^{\ast }\left( {[1+u]/\lambda } 
\right)=\lambda^{3}\int_0^1 {\Dot{{W}}^{\ast }\left( {[1+u(\tau )]/\lambda 
} \right)} d\tau ,\mbox{\, }0<x<1; \\  
{u}'(0)={u}'(1)=0; \qquad\int_0^1 {uds=0.} \\ 
 \end{array}
\end{equation}
\end{prop}
\begin{proof} If  $ u\in \mathcal{K}^{o} $  satisfies \eqref{eq9} then \eqref{eq11} holds, 
and by embedding,  $ u\in C[0,1]$.  Thus, \eqref{eq11} implies directly that the 
second distributional derivative of  $ u $ can be identified with a continuous 
function. An integration by parts in \eqref{eq11} then delivers
\begin{equation}
\label{eq13}
\varepsilon {u}''=\lambda^{3}\Dot{{W}}^{\ast }([1+u]/\lambda )-\mu 
,\mbox{\, }0<x<1,
\end{equation}
as well as the boundary conditions \eqref{eq12}$_{2} $. The nonlocal form 
\eqref{eq12}$_{1} $  follows by integrating \eqref{eq13} over  $ (0,1) $  while making use of 
the boundary conditions. The integral constraint \eqref{eq12}$_{3} $  is automatic, 
cf. \eqref{eq6}. Finally, \eqref{eq13} along with  $ u\in C[0,1] $  imply that  $ u $ is 
 $ C^{2} $ on the closed interval. The converse statement is obvious; starting 
with \eqref{eq12}, the above steps are reversible.  \end{proof}

Next, for given solution pair  $ (\lambda ,u)\in (0,\infty)\times 
\mathcal{K} $  of \eqref{eq9}, we define the closed \textit{broken set}
\begin{equation}
\label{eq14}
\mathcal{B}_{u} :=\{s\in [0,1]:u(s)=-1\},
\end{equation}
and the open \textit{glued set}
\begin{equation}
\label{eq15}
\mathcal{G}_{u} :=\{s\in [0,1]:u(s)>-1\}.
\end{equation}
Clearly  $ \mathcal{G}_{u} \cup \mathcal{B}_{u} =[0,1]$,  and in view of \eqref{eq5}, 
we note that $ \mathcal{G}_{u} \ne \emptyset$. 

\begin{thm}\label{thm2.3}Any solution $ u\in \mathcal{K},\lambda \in (0,\infty )$, of \eqref{eq9}  is continuously differentiable on  $ [0,1]$.  
Moreover, $ u\in C^{2}(\mathcal{G}_{u} )$.  
\end{thm}
\begin{proof}We choose  $ \psi ,\varphi \in H^{1}(0,1) $  as before leading 
to \eqref{eq11}, except we also require here that  $ \psi  \ge  0 $  and 
 $ h \ge  0 $  on  $ \mathcal{B}_{u}$.  Setting  $ v=u+h\in \mathcal{H} $  in 
\eqref{eq9} then leads to
\begin{equation}
\label{eq16}
\int_0^1 {[\varepsilon {u}'{\psi }'+\{\lambda^{3}\Dot{{W}}^{\ast }\left( 
{[1+u]/\lambda } \right)-\mu \}\psi ]ds \ge  0} \mbox{\, for\, all\, 
}\psi \in H^{1}(0,1)\mbox{\, with\, }\psi  \ge  0\mbox{\, on\, 
}\mathcal{B}_{u} ,
\end{equation}
where  $ \mu  $  is a constant multiplier. Specializing \eqref{eq16} to test functions 
 $ \psi  $ vanishing on  $ \mathcal{B}_{u} $,  we obtain
\begin{equation}
\label{eq17}
\int_{\mathcal{G}_{u} } {[\varepsilon {u}'{\psi }'+\{\lambda 
^{3}\Dot{{W}}^{\ast }\left( {[1+u]/\lambda } \right)-\mu \}\psi ]ds=0} 
\mbox{\, \, for\, all\, }\psi \in H^{1}(0,1)\mbox{\, with\, }\psi =0\mbox{\, 
on\, }\mathcal{B}_{u}.
\end{equation}
The same arguments used in the proof of Proposition \ref{prop2.2} now show that  $ u\in 
C^{2}(\mathcal{G}_{u} ) $  satisfies the differential equation
\begin{equation}
\label{eq18}
\varepsilon {u}''=\lambda^{3}\Dot{{W}}^{\ast }([1+u]/\lambda )-\mu \mbox{\, 
in\, }\mathcal{G}_{u} \mbox{.}
\end{equation}
A version of the Riesz representation theorem implies that the left side 
\eqref{eq16} is characterized by a non-negative Radon measure $ \mathfrak{m}$, which 
in view of \eqref{eq18}, has support in $ \mathcal{B}_{u}$, viz.,
\begin{equation}
\label{eq19}
\int_0^1 {[\varepsilon {u}'{\psi }'+\{\lambda^{3}\Dot{{W}}^{\ast }\left( 
{[1+u]/\lambda } \right)-\mu \}\psi ]ds=\int_{\mathcal{B}_{u} } {\psi 
d\mathfrak{m}(s)} } \mbox{\, for\, all\, }\psi \in H^{1}(0,1).
\end{equation}
Arguing as in  \cite{kinder}, we set  $ \varphi (s):=\mathfrak{m}([0,s))$, which is 
non-decreasing. Then \eqref{eq19} implies
\begin{equation}
\label{eq20}
\varepsilon {u}'(s)=\lambda^{3}\int_0^s {\Dot{{W}}^{\ast }([1+u(\xi 
)]/\lambda )d\xi -\mu s-\varphi (s)+C,} 
\end{equation}
where  $ C $  is a constant. Since  $ \varphi  $  is non-decreasing, \eqref{eq20} shows 
that  $ {u}'(s^{+}) \le  {u}'(s^{-})\mbox{\, on\, }[0,1]$.  Now 
consider $ s\in \mathcal{B}_{u} $, i.e.,  $ u(s)=-1$, and thus  $ u(t)-u(s) \ge  
-1-(-1)=0$.  Examination of the difference quotient  $ [u(t)-u(s)]/(t-s)(t\ne 
s) $  in the limit from the right and from the left as  $ t\to s $  reveals that 
 $ {u}'(s^{+}) \ge  0 \ge  {u}'(s^{-})$. \end{proof} 

\begin{rem}\label{rem2.4}An energy minimizer, say,  $ u_{\lambda } ,\lambda \in 
(0,\infty )$, given by Proposition \ref{prop2.1}, always satisfies \eqref{eq9}. As such, 
 $ u_{\lambda } \in C^{1}[0,1]$, by virtue of Theorem \ref{thm2.3}.
\end{rem}
Given that any solution  $ u $  of \eqref{eq9} is  $ C^{1} $  with  $ u\equiv -1 $  on 
 $ \partial \mathcal{G}_{u} $, integration of \eqref{eq18} as in the proof of 
Proposition \ref{prop2.2} yields 
\begin{cor}\label{cor2.5}{Given a solution pair } $ (\lambda ,u) $ {as in Theorem} \ref{thm2.3}, 
{then for any } $ (a,b)\subset \mathcal{G}_{u} 
 $ {with } $ a,b\in \partial \mathcal{G}_{u} $, 
\begin{equation}
\label{eq21}
\begin{array}{l}
 -\varepsilon {u}''+\lambda^{3}\Dot{{W}}^{\ast }\left( {[1+u]/\lambda } 
\right)=\frac{\lambda^{3}}{b-a}\int_a^b {\Dot{{W}}^{\ast }\left( {[1+u(\tau 
)]/\lambda } \right)} d\tau \mbox{\, in\, }\mathcal{G}_{u} \mbox{,} \\ 
 \mbox{\, \, \, \, \, \, \, \, \, \, \, \, \, \, \, 
}{u}'(a)={u}'(b)=0,\mbox{\, \, } \\ 
 \end{array}
\end{equation}
{and } $ u $ {is } $ C^{2} $ {on } $ [a,b]$. { If } $ [0,b)\subset \mathcal{G}_{u}  $ {with}  $ b\in \partial 
\mathcal{G}_{u} $, {then }\eqref{eq21} {holds with}  $ a=0$. { Likewise if } $ (a,1]\subset \mathcal{G}_{u}  $ {with} $ a\in 
\partial \mathcal{G}_{u} $, {then }\eqref{eq21} {holds with } $ b=1.  $ 
\end{cor}
Next, we obtain bounds on nontrivial solutions  $ u\ne 0,\lambda \in (0,1)$.  
Define  $ M>1  $ via  $ \Dot{{W}}^{\ast }(M)=\gamma :=\Dot{{W}}^{\ast }(0)$, cf. 
Fig.~\ref{Fig5a}.

\begin{thm}\label{thm2.6}{Any nontrivial solution pair of }\eqref{eq9},  $(\lambda ,u)\in (0,\infty )\times 
\mathcal{K},\mbox{\, }u\ne 0$, { is characterized by}
\begin{equation}
\label{eq22}
\begin{array}{l}
 \lambda >1/M, \\ 
 \left\| u \right\|_{\infty }  \le  C_{\lambda } ,\mbox{\, }\left\| u 
\right\|_{C^{1}} =\left\| u \right\|_{\infty } +\left\| {{u}'} 
\right\|_{\infty }  \le  C_{\lambda ,\varepsilon } , \\ 
 \end{array}
\end{equation}
{where } $ \left\| \cdot \right\|_{\infty }  $ { denotes the maximum norm over } $ [0,1]$, {and } $ C_{\lambda } ,C_{\lambda 
,\varepsilon }  $ { are constants having dependence upon the indicated subscripted quantities, but independent of } $ u.  $ 
\end{thm}
\bpr  Given that  $ u \ge  -1 $  is automatic, with  $ u\equiv -1 $  
on  $ \mathcal{B}_{u} $,  we only require upper bounds on nontrivial solutions 
 $ u\in \mathcal{G}_{u} $.  In view of Corollary \ref{cor2.5}, we consider any 
 $ (a,b)\subset \mathcal{G}_{u}  $ with  $ a,b\in \partial \mathcal{G}_{u} $, or 
 $ [0,b)\subset \mathcal{G}_{u}  $ with  $ b\in \partial \mathcal{G}_{u} $, or 
 $ (a,1]\subset \mathcal{G}_{u}  $ with  $ a\in \partial \mathcal{G}_{u} $. Since 
 $ u $ attains both its maximum and its minimum somewhere on the appropriate 
closed interval, there are points  $ s_{m} ,s_{M}  $  such that
\begin{equation}
\label{eq23}
\begin{array}{l}
 \max u(s)=u(s_{M} ),\mbox{\, }\min u(s)=u(s_{m} ), \\ 
 \mbox{\, \, \, \, \, \, \, with\, \, }u(s_{m} )<0<u(s_{M} ), \\ 
 \mbox{\, \, \, \, \, }{u}''(s_{m} ) \ge  0\mbox{\, and\, }{u}''(s_{M} 
) \le  0, \\ 
 \end{array}
\end{equation}
where \eqref{eq23}$_{2} $  is a consequence of \eqref{eq5}$_{1} $. If  $ s_{m}  $  and/or 
 $ s_{M}  $ are boundary points, then \eqref{eq23}$_{3\, } $ are one-sided derivatives, 
as justified in Corollary \ref{cor2.5}. We now evaluate \eqref{eq21} at  $ s_{m}  $ and at 
 $ s_{M} $, respectively, subtract the resulting equations and employ 
\eqref{eq23}$_{3} $  to deduce
\begin{equation}
\label{eq24}
\Dot{{W}}^{\ast }([1+u(s_{M} )]/\lambda ) \le  \Dot{{W}}^{\ast 
}([1+u(s_{m} )]/\lambda ).
\end{equation}
From the graph of  $ \Dot{{W}}^{\ast }(\cdot )$, cf. Fig.~\ref{Fig5}, it follows that 
both \eqref{eq23}$_{2\, } $ and \eqref{eq24} can be fulfilled only if
\begin{equation}
\label{eq25}
0<1/\lambda  \le  M\mbox{\, and\, }0 \le  [1+u(s_{m} )]/\lambda 
<[1+u(s_{M} )]/\lambda  \le  M,
\end{equation}
which is equivalent to
\begin{equation}
\label{eq26}
\lambda  \ge  1/M\mbox{\, and}-1 \le  u(s_{m} )<u(s_{M} ) \le  
-1+\lambda M
\end{equation}
This yields \eqref{eq22}$_{1,2} $, with  $ C_{\lambda } :=\max \{1,-1+\lambda M\}$.  
Combining \eqref{eq26} with \eqref{eq21}$_{1\, } $ (bootstrap) yields 
\eqref{eq22}$_{3} $.\epr
\begin{cor}\label{cor2.7} $ 0 \le  H(y)<M $  for all  $ y\in [0,\lambda ]$. \end{cor}
%

\section{Equilibria via Global Bifurcation}

In this section we consider the existence of other critical points via 
bifurcation from the trivial solution. In particular, note that  $ 0\in 
\mathcal{K}^{o}$, and thus, according to Proposition \ref{prop2.2}, the trivial 
solution branch  $ u\equiv 0 $ for all  $ \lambda \in (0,\infty )$, satisfies 
\eqref{eq12}. To begin, we consider the formal linearization of \eqref{eq12} at  $ u=0: $ 
\begin{equation}
\label{eq27}
\begin{array}{l}
 -\varepsilon {h}''(s)+\lambda^{2}\Ddot{{W}}^{\ast }(1/\lambda )h(s)=0,  \quad
0<s<1, \\ 
 {h}'(0)={h}'(1)=0,\mbox{\, }\int_0^1 h(s)ds =0, \\ 
 \end{array}
\end{equation}
where  $ \Ddot{{W}}^{\ast }:=\frac{d^{2}W^{\ast }}{dH^{2}}$.  Clearly \eqref{eq27} 
admits nontrivial solutions
\begin{equation}
\label{eq28}
h(s)=h_{n}(s) :=\cos (n\pi s),\mbox{\, }n=1,2,\ldots ,
\end{equation}
provided that the characteristic equation
\begin{equation}
\label{eq29}
\frac{\varepsilon n^{2}\pi^{2}}{\lambda^{2}}=-\Ddot{{W}}^{\ast }(1/\lambda 
),
\end{equation}
has corresponding roots  $ \lambda_{n} ,\mbox{\, }n=1,2,\ldots  $  For each 
value of  $ n\in \mathbb{N}$,  the left side of \eqref{eq29} defines a parabola in the 
variable  $1/\lambda$.  Then taking into account the graph of 
 $ -\Ddot{{W}}^{\ast }(\cdot )$,  cf. \eqref{eq-2} and Fig.~\ref{Fig5b}, we conclude that 
\eqref{eq29} has a countable infinity of simple transversal roots:
\begin{equation}
\label{eq30}
1/\kappa <\lambda_{1} <\lambda_{2} <\ldots ,
\end{equation}
each of which corresponds to a respective nontrivial solution \eqref{eq28}. 

\begin{figure}
  \centering
 \subfloat[]{\includegraphics[width=0.42\textwidth]{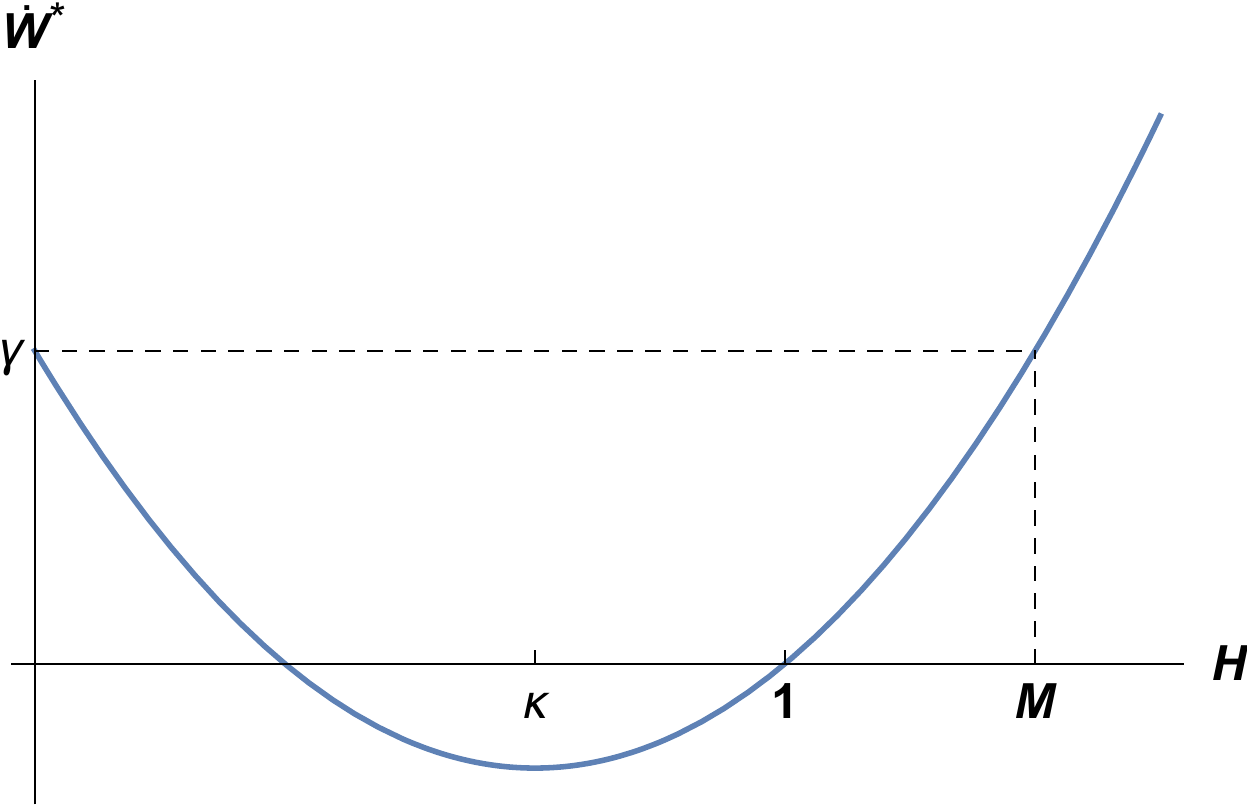}
  \label{Fig5a}}
  \hspace{0.8cm}
 \subfloat[]{\includegraphics[width=0.42\textwidth]{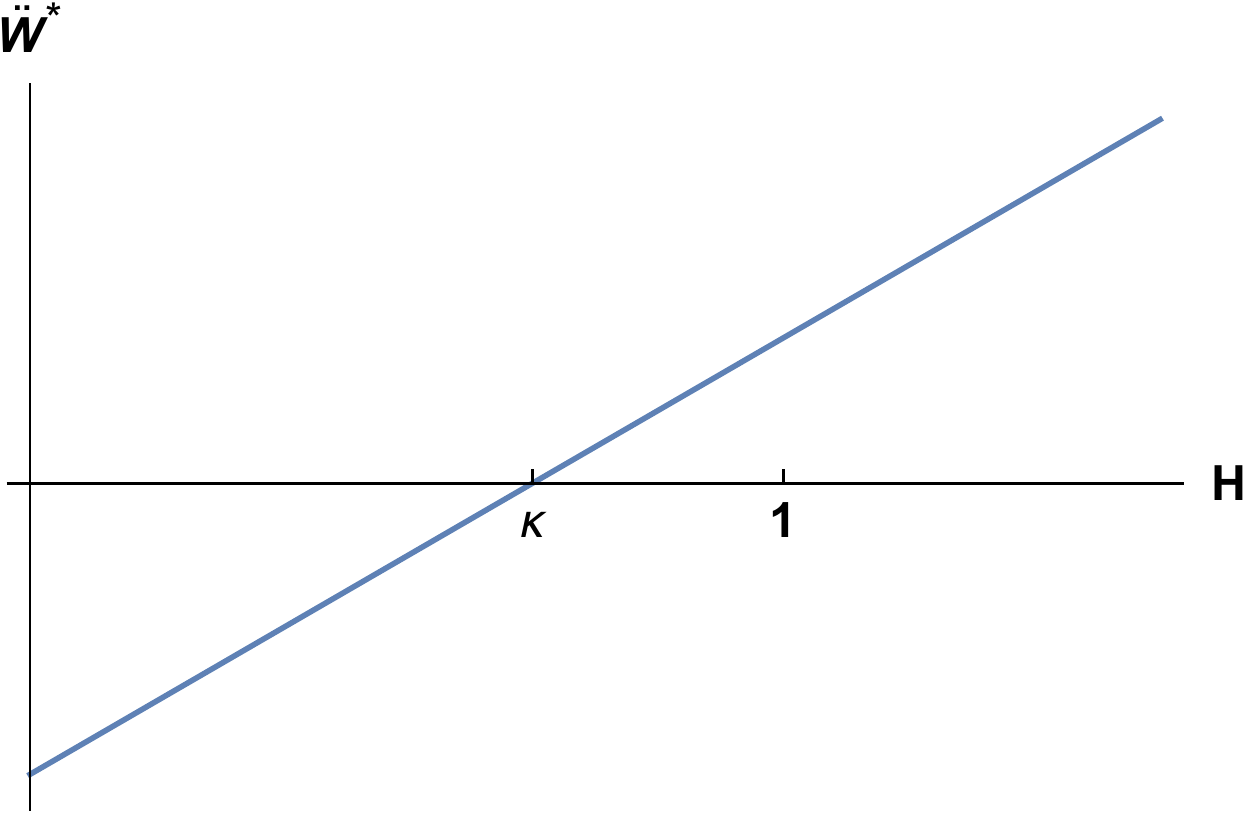}
  \label{Fig5b}} 
  \caption{(A) $\dot\sw$  \textit{vs}  $H$. (B)  $\ddot\sw$  \textit{vs}  $H$ .}
  \label{Fig5}
\end{figure}

To carry out a rigorous bifurcation analysis, we follow the approach in  \cite{globi}. 
We first express \eqref{eq9} abstractly via
\begin{equation}
\label{eq31}
\left\langle {\varepsilon u+G(\lambda ,u),v-u} \right\rangle  \ge  
0,\mbox{\, for\, all\, }v\in K,
\end{equation}
with  $ G:(0,\infty)\times \mathcal{H}\to \mathcal{H} $ defined by
\begin{equation}
\label{eq32}
\left\langle {G(\lambda ,u),v} \right\rangle :=\lambda^{3}\int_0^1 
{\Dot{{W}}^{\ast }\left( {[1+u(x)]/\lambda } \right)v(s)ds,\mbox{\, }} 
\end{equation}
for all  $ v\in \mathcal{H},\lambda \in (0,\infty )$.  Observe 
that $ \left\langle {G(\lambda ,0),v} \right\rangle \equiv 0 $  for all  $ v\in 
\mathcal{H},\lambda \in (0,\infty )$.  Since  $ \Dot{{W}}^{\ast } $  is 
continuous, the compact embedding  $ H^{1}(0,1)\to C[0,1] $  implies that  $ G $  is 
continuous and compact. Let  $ P_{\mathcal{K}}  $  denote the closest-point 
projection of  $ \mathcal{H} $  onto  $ \mathcal{K}$, i.e., for any  $ f\in 
\mathcal{H}, \quad u=P_{\mathcal{K}} f\Leftrightarrow u\in \mathcal{K} $  and 
 $ \left\langle {u-f,v-u} \right\rangle  \ge  0 $  for all  $ v\in 
\mathcal{K}$.  Then \eqref{eq31} is equivalent to the operator equation
\begin{equation}
\label{eq33}
\varepsilon u+F(\lambda ,u)=0,
\end{equation}
where  $ F:=P_{\mathcal{K}} \circ G$.  Due to the continuity of the 
projection $ P_{\mathcal{K}} $,  it follows that  $ F $  is continuous and compact 
on  $ \mathbb{R}\times \mathcal{K}$.  Thus, the Leray-Schauder degree of 
 $ u\mapsto u+\varepsilon^{-1}F(\lambda ,u) $  is well-defined.

Clearly  $ F(\lambda ,0)\equiv 0$,  and  $u\mapsto  F\equiv G $  in a neighborhood of  $ 0\in 
\mathcal{K}^{o}$.  By embedding, it follows that $ u\mapsto G(\lambda ,u) $  is 
differentiable on that same neighborhood, with Fr\'{e}chet derivative 
 $ A(\lambda )h:=D_{u} G(\lambda ,0)h=\lambda^{2}\Ddot{{W}}^{\ast 
}(1/\lambda )h $  for all  $ h\in \mathcal{H}$.  Hence, the rigorous 
linearization of \eqref{eq33} at the trivial solution reads 
\begin{equation}
\label{eq34}
\left\langle {\varepsilon h+A(\lambda )h,v} \right\rangle =0\mbox{\, for\, 
all\, }h,v\in \mathcal{H},
\end{equation}
and an integration by parts shows that \eqref{eq34} is equivalent to the formal 
linearization \eqref{eq27}. In particular, \eqref{eq28}-\eqref{eq30} imply that  $ (\lambda_{n} 
,0)\in (0,\infty )\times \mathcal{K}^{o},\mbox{\, }n=1,2,\ldots$, are potential 
bifurcation points of \eqref{eq33}. 

We observe that $ A(\lambda ):\mathcal{H}\to \mathcal{H} $  is compact. Also, 
for all  $ \lambda \ne \lambda_{n}$ and $n=1,2,\ldots$,  $u=0$ is the only 
solution of \eqref{eq33} on  $ B_{r} (0)\subset \mathcal{H}$, denoting some 
sufficiently small ball of radius  $ r>0 $  centered at the origin. Accordingly, 
the Leray-Schauder linearization principle shows that the topological degree 
of  $ u\mapsto u+\varepsilon^{-1}F(\lambda ,u) $  on  $ B_{r} (0)\subset 
\mathcal{H} $ is given by
\begin{equation}
\label{eq35}
\deg (I+\varepsilon^{-1}F(\lambda ,\cdot ),B_{r} (0),0)=(-1)^{m(\lambda 
)}\mbox{,}
\end{equation}
where  $ m(\lambda ) $  denotes the number of negative eigenvalues, counted by 
algebraic multiplicity, of the linear operator  $ I+\varepsilon^{-1}A(\lambda 
)$, for  $ \lambda \in \{ \lambda \in (0,\infty ):\lambda \ne \lambda_{n} ,\; n=1,2,\ldots\}$.  We 
can now state a global bifurcation existence result, cf.  \cite{rabin}: 

\begin{thm}\label{thm3.1} Let  $ \mathcal{S}\subset (0,\infty )\times K $ denote the closure of all nontrivial solution pairs $ (\lambda 
,u)\mbox{\, (}u\ne 0) $ of \eqref{eq33},{ and let } $ \mathcal{C}_{n}  $ denote the connected component of  $ \mathcal{S} $ containing $ (\lambda_{n} 
,0)$. Then   $(\lambda_{n} ,0),n=1,2,\ldots $, is a point of global bifurcation, viz.,  each solution branch  $ \mathcal{C}_{n}  $  is characterized by at least one of the following:

(i)  $ \mathcal{C}_{n}  $  is unbounded in $ (0,\infty)\times \mathcal{H}$,  

(ii) $ \{(\lambda_{k} ,0)\}\subset \mathcal{C}_{n} $,  $k\ne n$. 
\end{thm}

\bpr  We first observe that \eqref{eq22}$_{1\, } $ implies 
 $ \mathcal{C}_{n} \subset (1/M,\infty )\times \mathcal{K}$.  So 
it is enough to demonstrate that  $ \deg (I+\varepsilon^{-1}F(\lambda_{-} 
,\cdot ),B_{r} (0),0)\ne \deg (I+\varepsilon^{-1}F(\lambda_{+} ,\cdot 
),B_{r} (0),0)$,  for some $ \lambda_{-} \in (\lambda_{n} -\delta ,\lambda 
_{n} ) $  and  $ \lambda_{+} \in (\lambda_{n} ,\lambda_{n} +\delta )$,  for 
each  $ n=1,2,\ldots $, where  $ \delta >0 $  is sufficiently small. This is indeed the 
case here; the transversality of the roots \eqref{eq30} of \eqref{eq29} insures that 
 $ m(\lambda_{+} )=m(\lambda_{-} )+1 $  for each  $ n=1,2,\ldots $. Thus, the desired 
conclusion follows from \eqref{eq35}.\epr

By virtue of embedding, each solution branch  $ \mathcal{C}_{n} \subset 
(0,\infty )\times \mathcal{H},\mbox{\, }n=1,2,\ldots $, forms a continuum in 
 $ (0,\infty )\times C[0,1]$,  and for any  $ (\lambda ,u)\in \mathcal{C}_{n} 
$, we have that  $ u\in C^{1}[0,1]$, cf. Theorem \ref{thm2.3}. With that in hand, we now 
demonstrate that each branch is characterized by distinct nodal properties, 
generalizing a well-known result to our setting, cf.  \cite{cran,rabin}. Let 
 $ \mathcal{O}_{\ell }  $  denote the open set of all functions  $ v\in 
C^{1}[0,1] $  having precisely  $ \ell  $  zeros in  $ (0,1)$, each of which is 
simple, with  $ v(0)\ne 0 $  and  $ v(1)\ne 0$.  Let  $ \mathcal{C}_{n}^{o}  $  denote 
the component of  $ \mathcal{C}_{n} \cap (0,\infty )\times \mathcal{K}^{o} $  
containing the bifurcation point  $ (\lambda_{n} ,0), \quad n=1,2,\ldots  $  We call 
 $ \mathcal{C}_{n}^{o} \subset \mathcal{C}_{n}  $ a sub-branch. The argument 
given in \cite{cran,rabin} shows that nodal properties of solutions on global 
branches of 2$^{nd} $ -order ODE such as \eqref{eq12} are inherited from the 
corresponding eigenfunction \eqref{eq28} and can change only at the trivial 
solution. As such,  $ \mathcal{C}_{n}^{o} \backslash \{(\lambda_{n} 
,0)\}\subset (0,\infty )\times \mathcal{O}_{n} ,\mbox{\, }n=1,2,\ldots  $  In 
fact, this property holds for each of the global solution branches.

\begin{prop}\label{prop3.2}Each of the global bifurcating branches $ \mathcal{C}_{n}  $  of Theorem 3.1, is characterized by a distinct nodal pattern, viz.,
\begin{equation}
\label{eq36}
\mathcal{C}_{n} \backslash \{(\lambda_{n} ,0)\}\subset (0,\infty )\times 
\mathcal{O}_{n}, \quad n=1,2,\ldots 
\end{equation}
As such$, \mathcal{C}_{n} \cap \mathcal{C}_{m} =\emptyset $,  for all $ m\ne n$,  (alternative (ii) of Theorem 3.1 is not possible) and each  $ \mathcal{C}_{n} 
,n=1,2,\ldots , $  is unbounded in  $ \mathbb{R}\times \mathcal{H}$.  \end{prop}

\bpr  Suppose that \eqref{eq36} does not hold. Given that 
 $ \mathcal{C}_{n}^{o} \backslash \{(0,\lambda_{n} )\}\subset (0,\infty 
)\times \mathcal{O}_{n} ,\mbox{\, }n=1,2,\ldots , $  it follows that there 
exists a nontrivial solution pair $ (\lambda ,u)\in \mathcal{C}_{n} \backslash 
\mathcal{C}_{n}^{o}  $  with  $ u\in \partial \mathcal{O}_{n} $.  This means 
there is some  $ x_{o} \in \mathcal{G}_{u}  $ such that  $ u(x_{o} )={u}'(x_{o} 
)=0$.  But then by the uniqueness theorem for \eqref{eq12}, it follows that 
 $ u\equiv 0 $  on  $ [a,b]\subset \overline {\mathcal{G}_{u} }  $ for all such 
 $ (a,b)\subset \mathcal{G}_{u}  $ with  $ a,b\in \partial \mathcal{G}_{u} , $ or on 
 $ [0,b] $  with  $ b\in \partial \mathcal{G}_{u} , $ or on  $ [a,1] $  with  $ a\in 
\partial \mathcal{G}_{u} $.  Theorem \ref{thm2.3} together with \eqref{eq5}$_{1\, } $ then 
imply that  $ u\equiv 0 $ on  $ [0,1], $ i.e., nodal properties change only at the 
trivial solution. Finally, \eqref{eq36} along with the observation that 
 $ \mathcal{O}_{n} \cap \mathcal{O}_{m} =\emptyset , $  for all  $ m\ne n, $ imply 
that  $ \mathcal{C}_{n} \cap \mathcal{C}_{m} =\emptyset , $ for all  $ m\ne n, $  
and hence, alternative (ii) of Theorem 3.1 is not possible.  \epr

\section{Fracture and Stability}

\tred{We consider fractured solutions, which involve a nonempty broken set \eqref{eq14}, or at least one point $y\in[0,1]$ where $H(y)=0$. The onset of fracture occurs when the broken set consists of isolated points. From such solutions, we construct others where the broken set consists of one or more intervals, e.g., \eqref{eq44} below, where necessarily the inverse deformation is constant, and the associated original deformation is  discontinuous with opened cracks.}   

We demonstrate that fracture occurs on each of the bifurcating 
solution branches, and then obtain stability/instability results for some fractured 
solutions.  We find it convenient to return to the original variables via 
\eqref{eq3}, in which case the equilibrium equation \eqref{eq21} reads 
\begin{equation}
\label{eq37}
\begin{array}{l}
 -\varepsilon {H}''+\Dot{{W}}^{\ast }(H)=\varpi :=\frac{1}{\lambda_{\ast } 
}\int_{a_{\ast } }^{b_{\ast } } {\Dot{{W}}^{\ast }} (H(\tau ))d\tau \mbox{\, 
on\, }(a_{\ast } ,b_{\ast } ), \\ 
 \mbox{\, \, \, \, \, \, }{H}'(a_{\ast } )={H}'(b_{\ast } )=0, \\ 
 \end{array}
\end{equation}
where  $ a_{\ast } :=\lambda a,\mbox{\, }b_{\ast } :=\lambda b $ and  $ \lambda 
_{\ast } :=b_{\ast } -a_{\ast } $.  Clearly any solution of \eqref{eq21} gives a 
solution of \eqref{eq37}, and vice-versa. 

\begin{figure}
  \centering
 \subfloat[]{\includegraphics[width=0.42\textwidth]{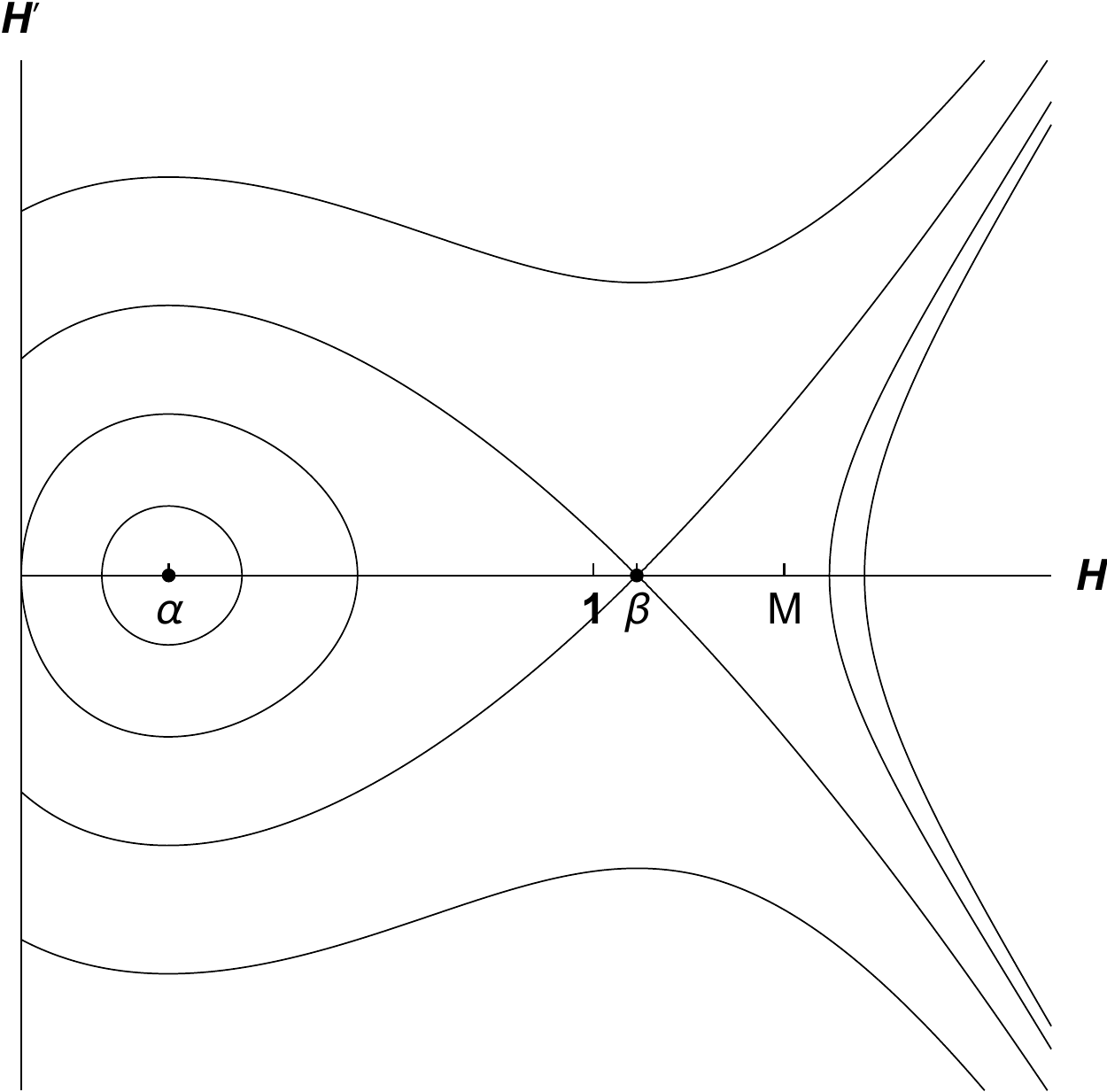}
  \label{Fig6a}}
  \hspace{0.8cm}
 \subfloat[]{\includegraphics[width=0.42\textwidth]{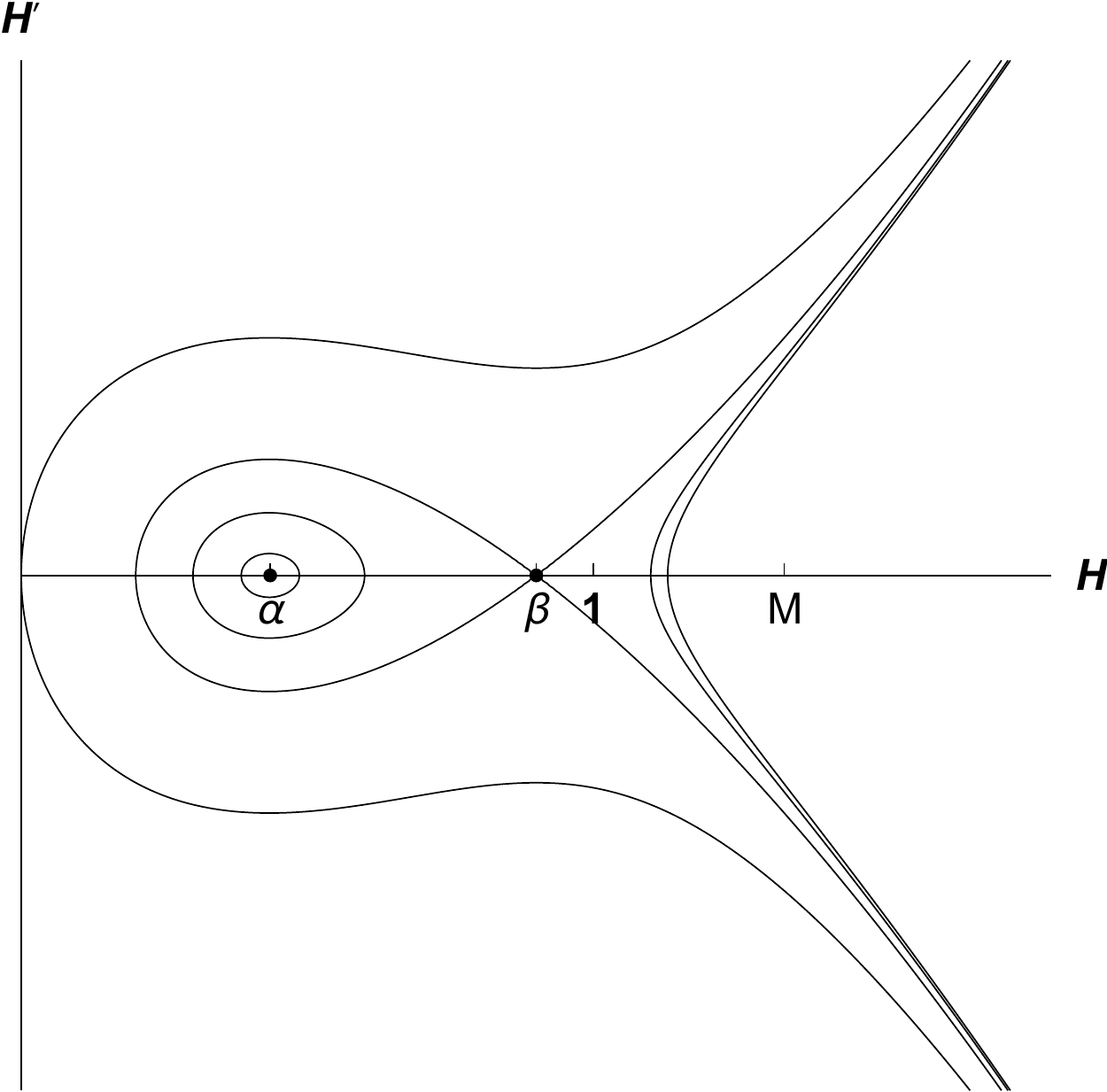}
  \label{Fig6b}} 
  \caption{Phase portraits for \eqref{eq39}: (A) $\dot\sw(\kappa)<\varpi\le0$  (B)  $0<\varpi<\gamma$ }
  \label{Fig6}
\end{figure}

In order to glean more information, we identify \eqref{eq37}$_{1} $  with a dynamical 
system, treating the independent variable  $ "y" $  as a time-like. The critical 
points or ``equilibria'' correspond to solutions of the algebraic equation
\begin{equation}
\label{eq38}
\Dot{{W}}^{\ast }(H)-\varpi =0.
\end{equation}
The graph of  $ \Dot{{W}}^{\ast } $ reveals that when  $ \Dot{{W}}^{\ast }(\kappa 
)<\varpi  \le  \gamma :=\Dot{{W}}^{\ast }(0), $ there are two solutions of 
\eqref{eq38}, denoted  $ H\equiv \alpha , $  and  $ H\equiv \beta , $ where  $ 0 \le  
\alpha <\beta  \le  M$.  cf. Fig.~\ref{Fig5a}. In addition, \eqref{eq37}$_{1} $  admits 
the first integral
\begin{equation}
\label{eq39}
\frac{\varepsilon }{2}({H}')^{2}-[W^{\ast }(H)-\varpi H]=\Gamma ,
\end{equation}
where  $ \Gamma  $  is a constant. With \eqref{eq39} in hand, we obtain the phase 
portraits depicted in Fig.~\ref{Fig6}, where  $ \alpha  $  is a ``center'', and  $ \beta 
 $  is a ``saddle''. According to the boundary conditions \eqref{eq37}$_{2} $, the 
trajectories should ``start and stop'' on the  $ H\mbox{-} $ axis. As such, we 
are interested only in closed orbits about the center
\begin{equation}
\label{eq40}
\alpha =1/\lambda ,
\end{equation}
such that  $ H \ge  0, $ cf. \eqref{eq2}$_{2} $. The right side of \eqref{eq40} follows 
from \eqref{eq3}, given that the former represents the trivial solution, viz., 
 $ H\equiv 1/\lambda \mbox{\, }\Leftrightarrow u\equiv 0$.  

\begin{rem}\label{rem4.1}In view of \eqref{eq40} and Fig.~\ref{Fig6}, it follows that 
 $ \lambda >1 $  for non-negative solutions  $ H \ge  0$, which improves \eqref{eq26}$_1$.  Referring again to 
\eqref{eq38}, the cases  $ \varpi  \le  \Dot{{W}}^{\ast }(\kappa ) $  and  $ \varpi 
>\gamma  $ are not associated with solutions of \eqref{eq37}. The first leads to 
either one degenerate critical point  $ \alpha =\beta  $  (for  $ \varpi 
=\Dot{{W}}^{\ast }(\kappa )) $  or no critical points, while the second case 
yields only one (positive) critical point. \end{rem}

According to Proposition \ref{prop2.2}, we infer that  $ (a_{\ast } ,b_{\ast } 
)=(0,\lambda ) $  in \eqref{eq37} for any strictly positive, non-constant solution 
 $ H>0, $  i.e., for  $ u\in \mathcal{K}^{o},\mbox{\, }u>-1, $ cf. \eqref{eq3}. Moreover, 
we must choose only those phase curves yielding even, functions of period 
 $ 2\lambda /\ell , $ where  $ \ell  $  is a positive integer. The first integral 
\eqref{eq39} implies that a non-constant, positive solution  $ H $ of \eqref{eq37} of period 
 $ 2\lambda /\ell  $  has either a maximum at  $ y=0 $  and a minimum at  $ y=\lambda 
/\ell , $  or vice-versa. In addition,  $ H $ is strictly monotone on  $ (0,\lambda 
/\ell ) $ and possesses reflection symmetry at the maximum and minimum 
locations, cf.  \cite{synge}. In view of \eqref{eq3}, these same qualitative properties hold 
for nontrivial solutions of \eqref{eq33} restricted to  $ (0,\infty )\times 
\mathcal{K}^{o}$.  

\begin{prop}\label{prop4.2}\textit{For any nontrivial solution pair } $ (\lambda ,u)\in \mathcal{C}_{n} \cap (0,\infty 
)\times \mathcal{K}^{o},n=1,2,\ldots , $ \textit{it follows that } $ u $ \textit{can be associated with a } $ C^{2}, $ \textit{ even, 2-periodic function. The resulting extension has minimal period } $ 2/n, $ \textit{ and either has a maximum at} $ s=0 $ \textit{and a minimum at } $ s=1/n $ \textit{ or vice-versa. Moreover, the extension possesses reflection symmetry about the maximum and minimum locations, and is strictly monotone on } $ (0,1/n)$.  
\end{prop}
\bpr  Given the equivalence of systems \eqref{eq21} and \eqref{eq37}, we see 
that  $ u $  possesses each of the appropriately scaled properties deduced 
above. So the only issue to address is the fact that  $ \ell =n$.  We argue by 
contradiction: Suppose that  $ \ell \ne n$.  Now \eqref{eq5}$_{1} $  insures that 
 $ u(0) $ and  $ u(1/\ell ) $ have opposite signs. Thus,  $ u $  has a single zero on 
 $ \left( {0,1/\ell } \right), $ leading to precisely  $ \ell  $  zeros on  $ (0,1)$.  
But this contradicts \eqref{eq36} unless  $ \ell =n$. \epr

Recall that the signature of fracture is  $ H=0 $ at some point or points in 
 $ [0,\lambda ]$.  Referring again to \eqref{eq39} and the phase portrait in Fig.~
5(b), we see that this corresponds to a closed phase curve containing the 
origin. As in the previous section, the sub-branch  $ \mathcal{C}_{n}^{o}  $  is 
defined as the component of  $ \mathcal{C}_{n} \cap (0,\infty )\times 
\mathcal{K}^{o} $  that contains the bifurcation point  $ (\lambda_{n} 
,0),n=1,2,\ldots  $  

\begin{thm}\label{thm4.3}\textit{Each} \textit{of the sub-branches, } $ \mathcal{C}_{n}^{o} ,n=1,2,\ldots , $ is bounded 
in  $ (0,\infty )\times \mathcal{H}$. \end{thm}

\bpr  We argue by contradiction; suppose that 
 $ \mathcal{C}_{n}^{o}  $  is unbounded. The bound \eqref{eq22}$_{3} $  then implies 
there is a sequence  $ \left\{ {(\lambda_{j} ,u_{j} )} \right\}\subset 
\mathcal{C}_{n}^{o}  $  with  $ \lambda_{j} \to \infty $.  By \eqref{eq3} we then have 
a sequence of solutions  $ \left\{ {(\lambda_{j} ,H_{j} )} \right\} $  of \eqref{eq37} 
on  $ (a_{\ast } ,b_{\ast } )=(0,\lambda_{j} ) $  such that
\begin{equation}
\label{eq41}
H_{j} (y)=[1+u_{j} (y/\lambda_{j} )]/\lambda_{j} >0\mbox{\, on\, 
}[0,\lambda_{j} ],
\end{equation}
with the sequence of centers \eqref{eq40} satisfying
\begin{equation}
\label{eq42}
\alpha_{j} =1/\lambda_{j} \searrow 0.
\end{equation}
Since  $ H_{j}  $  is positive, \eqref{eq42} implies that the amplitude of the 
``oscillation'' must also approach zero. Indeed, set  $ 0<\delta_{j} :=
\min 
\limits_{y\in [0,\lambda ]} H^{j}(y)<\alpha_{j} =1/\lambda_{j} $.  The 
phase curve \eqref{eq39} containing the point  $ (0,\delta_{j} ) $ is then given by
\begin{equation}
\label{eq43}
\frac{\varepsilon }{2}({H}')^{2}+\varpi_{j} H-W^{\ast }(H)=\Gamma_{j} ,
\end{equation}
where  $ \Gamma_{j} =\varpi_{j} \delta_{j} -W^{\ast }(\delta_{j} ) $  and 
 $ \varpi_{j} =\Dot{{W}}^{\ast }(\alpha_{j} ), $ cf. \eqref{eq38} and Fig.~\ref{Fig5a}. 
From \eqref{eq38} and \eqref{eq42}, we have  $ \varpi_{j} =\Dot{{W}}^{\ast }(\alpha_{j} 
)\nearrow \Dot{{W}}^{\ast }(0)=\gamma  $  and  $ W^{\ast }(\delta_{j} )\to 
W^{\ast }(0)=0$.  Thus,  $ \Gamma_{j} \searrow 0 $ as  $ j\to \infty $.  But \eqref{eq39} 
with  $ \varpi =\gamma  $  and  $ \Gamma =0 $  gives only the ``equilibrium'' 
 $ H\equiv 0$.  Thus, for all sufficiently large  $ j, $ \eqref{eq43} yields a small 
closed orbit indicating a bounded minimal period. But this contradicts the 
previous observation (above Proposition 4.2) that the minimal period of the 
solution  $ H_{j}  $  is given by  $ T_{j} =2\lambda_{j} /n\mbox{\, }\to \infty  $  
as  $ j\to \infty $. \epr  

Proposition 3.2 and Theorem 4.3 immediately imply  $ \mathcal{C}_{n}^{o} \ne 
\mathcal{C}_{n} ,n=1,2,\ldots , $ and we also deduce

\begin{cor}\label{cor4.4}\textit{The closure of each sub-branch contains a fractured solution, viz., there is } $ (\lambda_{n}^{\ast } ,u_{n}^{\ast } )\in \overline 
{\mathcal{C}_{n}^{o} }  $ \textit{ with } $ u_{n}^{\ast } \in \partial \mathcal{K} $ \textit{ and } $ \lambda 
_{n}^{\ast } >1, $ \textit{ for } $ n=1,2,\ldots $. \end{cor}

\bpr  Consider sequences  $ \left\{ {(\lambda_{j} ,u_{j} )} 
\right\}\subset \mathcal{C}_{n}^{o}  $  and  $ \left\{ {(\Tilde{{\lambda }}_{j} 
,\Tilde{{u}}_{j} )} \right\}\subset \mathcal{C}_{n} \backslash 
\mathcal{C}_{n}^{o} , $ each converging to  $ (\lambda_{n}^{\ast } ,u_{n}^{\ast 
} )\in \mathcal{C}_{n} $.  Then $ (\lambda_{n}^{\ast } ,u_{n}^{\ast } )\in 
\overline {\mathcal{C}_{n}^{0} } \cap (\mathcal{C}_{n} \backslash 
\mathcal{C}_{n}^{0} )$.  The first observation made in Remarks 4.1 shows 
that $ \lambda_{n}^{\ast } >1$. \epr   

For each  $ (\lambda_{n}^{\ast } ,u_{n}^{\ast } )\in \overline 
{\mathcal{C}_{n}^{o} }  $  as above, it is clear that  $ u_{n}^{\ast }  $  is also 
characterized as in Proposition 4.2, with a minimum value of  $ -1$.  Let 
 $ H_{n}^{\ast }  $  denote its rescaling via \eqref{eq3}, which corresponds to a 
closed phase curve containing the origin in Fig.~\ref{Fig6b}. In view of 
Proposition 3.2,  $ H_{n}^{\ast }  $  can be associated with a  $ C^{2} $  even 
function of period  $ 2\lambda_{n}^{\ast } /n, $  having either a maximum at 
 $ y=0 $  and a minimum at  $ y=\lambda_{n}^{\ast } /n $  or vice-versa. In 
addition,  $ H_{n}^{\ast }  $  is strictly monotone on  $ (0,\lambda_{n}^{\ast } 
/n) $  and possesses reflection symmetry at the maximum and minimum locations. 
Observe that fracture  $ (H_{n}^{\ast } =0) $  occurs only at a finite number of 
points in  $ [0,\lambda_{n}^{\ast } ]$.  

We claim that all other fractured or broken solutions for  $ \lambda >\lambda 
_{n}^{\ast } , $  denoted  $ H_{n,\lambda } , $  can be constructed from 
 $ H_{n}^{\ast }  $  via a ``cut and paste'' procedure:  $ H_{n,\lambda } \equiv 
0 $  is inserted on a set of measure  $ \lambda -\lambda_{n}^{\ast }  $  at 
locations where  $ H_{n}^{\ast }  $  possesses zeros, while  $ H_{n,\lambda } 
 $ maintains the values of  $ H_{n}^{\ast } , $  but now translated on  $ n $  
sub-intervals intervals, each of length  $ \lambda_{n}^{\ast } /n$.  There are 
only two such possibilities when  $ n=1 $. This follows from the fact that 
 $ H_{1}^{\ast }  $  has only one zero on  $ [0,\lambda_{\ast } ], $ which, due to 
monotonicity, occurs at one of the two ends of the interval. For example, if 
 $ H_{1}^{\ast } (\lambda_{1}^{\ast } )=0, $ then
\begin{equation}
\label{eq44}
H_{1,\lambda } (y)=\begin{cases}
 H_{1}^{\ast } (y), & y\in [0,\lambda_{1}^{\ast} ], 
\\
0, & y\in (\lambda_{1}^{\ast} ,\lambda].
\end{cases} 
\end{equation}

For each  $ n \ge  2, $  there are multiple locations where open sets of 
total measure  $ \lambda -\lambda_{n}^{\ast }  $  can be inserted. Hence, there 
are uncountably many, measure-theoretic equivalent ways to carry out the 
latter construction. We now show that the above construction is rigorous. 

\begin{thm}\label{thm4.5}\textit{If } $ (\lambda ,u)\in \mathcal{C}_{n} \backslash \overline 
{\mathcal{C}_{n}^{o} } ,n=1,2,\ldots , $ \textit{ then there are precisely } $ n $ \textit{ open sub-intervals, each of length } $ \lambda_{n}^{\ast } /n\lambda , \lambda 
>\lambda_{n}^{\ast } , $ \textit{ over which } $ u $ \textit{ is monotone with } $ u>-1 $ \textit{, and } $ u\equiv -1 $ \textit{on the complementary set of measure } $ 1-(\lambda_{n}^{\ast } 
/\lambda ), $ \textit{where } $ (\lambda_{n}^{\ast } ,u_{n}^{\ast } )\in \overline 
{\mathcal{C}_{n}^{o} }  $ \textit{ is defined in Corollary \ref{cor4.4}. Moreover, if the broken part of the solution, viz., } $ u\equiv -1, $ \textit{ is excised, then } $ u $ \textit{can be associated with a } $ C^{2} $ \textit{ even, } $ 2/n $ \textit{-periodic function, viz., with } $ u_{n}^{\ast } .  $ \end{thm}

\bpr Consider the proposed construction of a broken solution as 
described above. Then  $ H_{n}^{\ast }  $  satisfies \eqref{eq37} on each of the  $ n $  
sub-intervals, of length  $ \lambda_{n}^{\ast } /n, $  contained in  $ [0,\lambda 
], $  with  $ H_{n,\lambda } \equiv 0 $  on the complement. Rescaling according to 
\eqref{eq3} then then yields
\[
u_{n,\lambda } (s):=\begin{cases}
\lambda \Tilde{{H}}_{n}^{\ast } (\lambda s)-1, & s\in 
\mathcal{G},  \\
-1, &s\in 
\mathcal{B},\\
\end{cases} 
\]
where  $ \left| \mathcal{B} \right|=1-(\lambda_{n}^{\ast } /\lambda ), $  
 $ \mathcal{G} $  is the union of  $ n $  intervals, each of length  $ \lambda 
_{n}^{\ast } /n\lambda , $  and  $ \Tilde{{H}}_{n}^{\ast }  $  represents the 
translated portions of  $ \lambda H_{n}^{\ast } (\lambda \cdot )-1 $  on 
 ${}{} \mathcal{G}_u$.  For example, the rescaling of \eqref{eq44} reads
\begin{equation}
\label{eq45}
u_{1,\lambda } (s)=\begin{cases}
\lambda H_{1}^{\ast } (\lambda s)-1, & s\in [0,\lambda 
_{1}^{\ast } /\lambda ],  \\
-1, & s\in (\lambda_{1}^{\ast } /\lambda ,1\mbox].\\
\end{cases} 
\end{equation}
Clearly  $ u_{n,\lambda } \in \mathcal{K} $  satisfies \eqref{eq21} on each of the  $ n $  
sub-intervals comprising  $\mathcal{G}, $  with  $ b-a=\lambda_{n}^{\ast } 
/n\lambda  $  and  $ \mu =\lambda^{3}\varpi $.  We claim that  $ u_{n,\lambda }  $  
satisfies \eqref{eq16}, which is equivalent to \eqref{eq9}: First decompose the left 
side of \eqref{eq16} evaluated at  $ u_{n,\lambda }  $ into the sum of two integrals 
over $ \mathcal{G} $  and  $ \mathcal{B}$.  The former vanishes by construction, 
and the integral over  $ \mathcal{B} $ leads directly to 
\begin{equation}
\label{eq46}
(\lambda^{3}\Dot{{W}}^{\ast }(0)-\mu )\int_\mathcal{B} \psi dx=\lambda 
^{3}(\gamma -\varpi )\int_\mathcal{B} \psi dx.
\end{equation}
From the development in Section 4, we have  $ \gamma >\varpi , $ cf. Fig.~\ref{Fig6}. 
Accordingly \eqref{eq46} is positive for all test functions  $ \psi \in 
H^{1}(0,1) $ with  $ \psi  \ge  0 $  on  $ \mathcal{B}, $ which proves the claim. 
Clearly  $ (\lambda ,u_{n,\lambda } ) $ is connected to  $ (\lambda_{n}^{\ast } 
,u_{n}^{\ast } ), $  and thus,  $ (\lambda ,u_{n,\lambda } )\in \mathcal{C}_{n} 
 $  for all  $ \lambda >\lambda_{n}^{\ast } $.  Finally, we can associate each 
such  $ u_{n,\lambda }  $ with  $ u_{n}^{\ast } (s)=\lambda H_{n}^{\ast } (\lambda 
s)-1 $  on  $ \mathbb{R}(\bmod 2/n) $  via excision.   \epr   

We finish this section with some stability/instability results. We adopt the 
usual energy criterion for stability, viz., an equilibrium solution is 
\textit{stable} if it renders the potential energy a minimum; it is \textit{locally stable} if it renders the 
potential energy a local minimum; if the potential energy is not a minimum 
(neither local not global), the equilibrium is unstable. Referring to \eqref{eq4}, 
it is not hard to show that  $ V_{\varepsilon } [\lambda ,\cdot ] $  is a  $ C^{2} $  
functional on  $ \mathcal{H}$.  Thus, we may rigorously employ the 
second-derivative test (second variation) to determine local minima or to 
demonstrate instability. For any solution of \eqref{eq9} or equivalently \eqref{eq33}, 
the second variation takes the form
\begin{equation}
\label{eq47}
\delta^{2}V_{\varepsilon } [\lambda ,u;\eta ]:=\frac{d^{2}}{d\tau 
^{2}}V_{\varepsilon } [\lambda ,u+\tau \eta ]\vert_{\tau =0} =\int_0^1 
{[\varepsilon ({\eta }')^{2}} +\lambda^{2}\Ddot{{W}}^{\ast }([1+u]/\lambda 
)\eta^{2}]ds,
\end{equation}
for all  $ \eta \in \mathcal{H}$.  We first consider the trivial solution.

\begin{prop}\label{prop4.6}\textit{The trivial solution}  $ u\equiv 0 $ \textbf{\textit{ 
}}\textit{is locally stable for all } $ \lambda \in (0,\lambda_{1} ) $ \textit{and unstable for all } $ \lambda \in (\lambda_{1} ,\infty ), $ \textit{ where } $ \lambda 
_{1}  $ \textit{ is the smallest root of \eqref{eq29}, i.e., } $ (\lambda_{1} ,0)\in (0,\infty )\times \mathcal{H} $ \textit{is the first bifurcation point, cf. Theorem \ref{thm3.1}. } 
\end{prop}
\bpr  Evaluating \eqref{eq47} at  $ u\equiv 0 $ gives
\begin{equation}
\label{eq48}
\delta^{2}V_{\varepsilon } [\lambda ,0;\eta ]=\int_0^1 {[\varepsilon ({\eta 
}')^{2}} +\lambda^{2}\Ddot{{W}}^{\ast }(1/\lambda )\eta^{2}]ds,
\end{equation}
for all variations  $ \eta \in \mathcal{H}$.  Next we employ the sharp 
Poincar\'{e} inequality
\begin{equation}
\label{eq49}
\int_0^1 {({\eta }')^{2}} ds \ge  \pi^{2}\int_0^1 {\eta^{2}} ds,
\end{equation}
where  $ \pi^{2} $  is the first eigenvalue of the operator  $ -{\eta }'' $  on 
 $ [0,1], $  subject to conditions \eqref{eq27}$_{2,3} $. Then \eqref{eq48}, \eqref{eq49} lead to
\begin{equation}
\label{eq50}
\delta^{2}V_{\varepsilon } [\lambda ,0;\eta ] \ge  [\varepsilon \pi 
^{2}+\lambda^{2}\Ddot{{W}}^{\ast }(1/\lambda )]\int_0^1 {\eta^{2}} ds.
\end{equation}
In view of \eqref{eq29} with $ n=1, $ and from the graph of  $ -\Ddot{{W}}^{\ast }(\cdot 
), $ we conclude that  $ \varepsilon \pi^{2}/\lambda^{2}>-\Ddot{{W}}^{\ast 
}(1/\lambda ) $  for all  $ 1/\lambda >1/\lambda_{1} , $  cf. (1.4) and Fig.~\ref{Fig5b}. 
Of course this implies that \eqref{eq50} is positive for all  $ \lambda \in 
(0,\lambda_{1} ) $ and  $ \eta \in \mathcal{H}$. 

For any  $ \lambda >\lambda_{k,}  $  choose the admissible test functions  $ \eta 
_{\ell } :=\cos (\ell \pi s), $ for  $ \ell =1,2,\ldots ,k$.  Then
\[
\delta^{2}V_{\varepsilon } [\lambda ,0;\eta_{\ell } ]=[\varepsilon \ell 
^{2}\pi^{2}+\lambda^{2}\Ddot{{W}}^{\ast }(1/\lambda )]/2.
\]
Again, from \eqref{eq29} and the graph of  $ \Ddot{{W}}^{\ast }(\cdot ), $ we see that 
 $ \varepsilon \ell^{2}\pi^{2}/\lambda^{2}<-\Ddot{{W}}^{\ast }(1/\lambda 
), $  implying that  $ \delta^{2}V_{\varepsilon } [\lambda ,0;\eta_{\ell } 
]<0, $ for  $ \ell =1,2,\ldots ,k$. \epr  

Next, we borrow a construction from  \cite{cgs} to show that the ``higher-mode'' 
global solution branches are all unstable. 

\begin{prop}\label{prop4.7}For any  $ (\lambda ,u)\in \mathcal{C}_{n}$, $n=2,3,\ldots$,
the second variation \eqref{eq47} is strictly negative at some  $ \eta \in \mathcal{H}$,  i.e., $u$  is unstable. 
\end{prop}
\bpr In view of Proposition 4.6, we need only consider  $ u\ne 0$.  
Let  $ u $  be associated with an even,  $ (2\lambda_{\ast } /n\lambda 
)- $ periodic function: If  $ u\in \mathcal{K}^{o}, $ then  $ \lambda_{\ast } 
=\lambda , $  cf. Proposition 4.2. If  $ u\notin \mathcal{K}^{o}, $ then  $ \lambda 
_{\ast } =\lambda_{n}^{\ast } , $ cf. Theorem 4.5. In any case, for 
 $ n \ge  2 $ there is an interval  $ (a,b)\subset (-1,1), $ with  $ b-a=2\lambda 
_{\ast } /n\lambda , $  such that \eqref{eq21} is valid. Theorem 4.5 implies that 
 $ u $  is reflection-symmetric with respect to the midpoint of  $ (a,b)$.  Without 
loss of generality, we assume  $ a=0$.  Note that \eqref{eq21} implies (via 
bootstrap) that
\begin{equation}
\label{eq51}
\varepsilon {u}'''\equiv \lambda^{2}\Ddot{{W}}^{\ast }\left( {[1+u]/\lambda 
} \right){u}'.
\end{equation}
Next, define
\[
\varphi (s):=\begin{cases}
 u'(s), & s\in [0,2\lambda^{\ast }/n\lambda ],  \\
 0,&  \hbox{otherwise.}
 \end{cases} 
\]
By virtue of Theorem 4.5, we note that  $ {u}' $  can be associated with an odd, 
 $ (2\lambda_{\ast } /n\lambda )$-periodic function. Accordingly,  $ \int_0^1 
{\varphi ds=0}$, and thus  $ \varphi \in \mathcal{H}$.  Now define the 
admissible variation
\begin{equation}
\label{eq52}
\eta (s):=\phi (s)+\tau \psi (s),
\end{equation}
where  $ \psi  $  is any function in  $ \mathcal{H} $  such that  $ \psi (0)=1,\psi 
\equiv 0 $  on  $ [2\lambda_{\ast } /n\lambda ,1], $  and  $ \tau  $  is a small 
parameter. On substituting \eqref{eq52} into \eqref{eq47}, we obtain
\[
\begin{array}{l}
 \delta^{2}V_{\varepsilon } [\lambda ,u;\eta ]=\int_0^b {[\varepsilon 
({u}'')^{2}} +\lambda^{2}\Ddot{{W}}^{\ast }([1+u]/\lambda )({u}')^{2}]ds \\ 
 \mbox{\, \, \, \, \, \, \, \, \, \, \, \, \, \, \, \, \, \, \, \, \, \, \, 
\, \, }+2\tau \int_0^b {[\varepsilon {u}''{\psi }'+} \lambda 
^{2}\Ddot{{W}}^{\ast }([1+u]/\lambda ){u}'\psi ]ds+O(\tau^{2}), \\ 
 \end{array}
\]
where  $ b=2\lambda_{\ast } /n\lambda $.  Then employing \eqref{eq51}, we find
\[
\delta^{2}V_{\varepsilon } [\lambda ,u;\eta ]=-2\varepsilon \tau 
{u}''(0)+O(\tau^{2}).
\]
Finally, since  $ u $  is a non-constant solution, we have  $ {u}''(0)\ne 0$.  
Thus, \eqref{eq47} is negative for  $ \left| \tau \right| $  sufficiently small. 
  \epr  

\section{Effective Macroscopic Behavior}

In this section we interpret our results in terms of the conventional 
Lagrangian description, as discussed in Section 1. Taking the point of view 
of  \cite{lifsh}, our goal is to obtain the effective or macroscopic stress-stretch 
diagram based on the global solutions obtained. This is an alternative 
global bifurcation diagram. In view of Propositions 4.6 and 4.7, it is enough 
to consider only the trivial solution and the first global branch 
 $ \mathcal{C}_{1} $. 

Presuming  $ H>0, $  we first express \eqref{eq1}, \eqref{eq2}$_{1} $  in terms of the 
deformation gradient
\begin{equation}
\label{eq53}
F(x):={f}'(x)=1/H(f(x)).
\end{equation}
Recalling  $ y=f(x)\Leftrightarrow x=h(y), $  we likewise have
\begin{equation}
\label{eq54}
H(y)={h}'(y)=1/F(h(y)).
\end{equation}
By the chain rule and the change-of-variable formula, we find that
\begin{equation}
\label{eq55}
{H}'(f(x))=-{F}'(x)/(F(x))^{3},
\end{equation}
and the total energy becomes \cite{carlson}
\begin{equation}
\label{eq56}
\Tilde{{E}}_{\varepsilon } [F]=\int_0^1 \left[\frac{\varepsilon 
}{2}\frac{({F}')^{2}}{F^{5}}+W(F)\right]dx,
\end{equation}
subject to
\begin{equation}
\label{eq57}
\int_0^1 {Fdx=\lambda .} 
\end{equation}
The Euler-Lagrange equation for \eqref{eq56}, \eqref{eq57} is readily obtained:
\begin{equation}
\label{eq58}
-\varepsilon \left[ {\left( {\frac{{F}'}{F^{5}}} \right)^{\prime 
}+\frac{5}{2}\left( {\frac{{F}'}{F^{3}}} \right)^{2}} 
\right]+\Dot{{W}}(F)=\sigma \mbox{\, on\, (0,1).}
\end{equation}
where the multiplier  $ \sigma , $  enforcing \eqref{eq57}, represents the constant 
stress carried by the bar. Indeed, along the homogeneous 
solution  $ F\equiv \lambda  $  we obtain the first-gradient constitutive law
\begin{equation}
\label{eq59}
\sigma =\Dot{{W}}(\lambda ),
\end{equation}
whose graph, in view of (1.3) and Fig.~\ref{Fig2a}, is depicted (blue curve) in Fig.~\ref{Fig4}. Using 
\eqref{eq53}-\eqref{eq55} and the chain rule, it is not hard to see that \eqref{eq58} is 
equivalent to
\begin{equation}
\label{eq60}
\sigma =\varepsilon [{H}''H-({H}')^{2}/2]+W^{\ast }(H)-H\Dot{{W}}^{\ast 
}(H)\mbox{\, on\, (0,}\lambda \mbox{),}
\end{equation}
where we have also employed (1.1). From \eqref{eq37} we deduce 
\begin{equation}
\label{eq61}
\varepsilon {H}''H=H[\Dot{{W}}^{\ast }(H)-\varpi ],
\end{equation}
and this together with \eqref{eq60} shows that the constant of integration 
appearing in \eqref{eq39} is, in fact, the negative of the stress, 
viz.,  
\beq\label{gs}\sigma  =- \Gamma. \eeq 

\begin{thm}\label{thm5.1} If$ (\lambda ,u)\in \mathcal{C}_{1} \backslash 
\mathcal{C}_{1}^{o}$, then the stress $\si$ in \eqref{eq60} vanishes on that solution, i.e., the fractured bar carries no stress.\end{thm}

\bpr Recall from the discussion above \eqref{eq44} that fracture occurs 
at one of the two ends of the bar; without loss of generality, we presume 
that  $ u $ is characterized by \eqref{eq44}, \eqref{eq45}. Then  $ u(\lambda_{1}^{\ast } 
/\lambda )=-1\Leftrightarrow H_{1}^{\ast } (\lambda_{1}^{\ast } )=0, $  and 
 $ u^{\prime }(\lambda_{1}{}^{\ast } /\lambda )=0\Leftrightarrow H_{1}^{\ast}{}^{\prime }(\lambda_{1}^{\ast } )=0$.  Now \eqref{eq39} is valid for 
 $ 0 \le  \lambda  \le  \lambda_{1}^{\ast }  $ ; in particular, at 
 $ y=\lambda_{1}^{\ast } , $ we obtain  $ \Gamma =-\sigma =W^{\ast }(0)-\varpi 
\cdot 0=0, $ cf. (1.5).    \epr  

The bounded solution branch  $ \overline {\mathcal{C}_{1}^{o} }  $ connects the 
bifurcation point  $ (\lambda_{1} ,0) $  to the fracture point  $ (\lambda_{\ast 
} ,u_{\ast } ):=(\lambda_{1}^{\ast } ,u_{1}^{\ast } ), $ cf. Corollary \ref{cor4.4}. 
  \tred{Hence, with  $ \lambda 
=f(1) $  playing the role of macroscopic stretch, the projection of  $ \overline 
{\mathcal{C}_{1}^{o} }  $  onto the  $ (\lambda ,\sigma ) $  plane,
\beq\label{slt} \mathcal{T}=\{(\la,\sigma)\in  \mathbb{R}^2: (\la,u)\in \overline {\mathcal{C}_{1}^{o} } :  \hbox{$\sigma$ satisfies \eqref{eq58} for $H$ as in \eqref{eq3}}\} \eeq
connects 
 $ (\lambda_{1} ,\Dot{{W}}(\lambda_{1} )) $  to  $ (\lambda_{\ast } ,0)$.  An example for a specific choice of $\sw$ (details at the end of this section) is shown in Fig.~\ref{Fig4}, in which, for $\ep=2/49$, $\mathcal{T}$ is the orange curve, connecting the orange bifurcation point $ (\lambda_{1} ,\Dot{{W}}(\lambda_{1} )) $ to the black fracture point $ (\lambda_{\ast } ,0)$. The 
remaining part  $ \mathcal{C}_{1} \backslash \mathcal{C}_{1}^{o}  $  consists of broken (fractured) solutions of the form \eqref{eq44} with an opened crack; it projects 
onto the ray 
 \beq\label{sltt} \mathcal{F} =\{(\lambda,0):\lambda \ge \lambda_{\ast } \}. \eeq  
 In the example of Fig.~\ref{Fig4}, for $\ep=2/49$, $\mathcal{F}$ is the part of the horizontal axis to the right of the black point.}

We can say something about the stability of the component of the first branch consisting of broken solutions at this 
general stage. 
\begin{thm}\label{thm5.2} There exits  $ \lambda_{m} <\infty  $,  such that, for every $ (\lambda ,u)\in 
\mathcal{C}_{1} \backslash \mathcal{C}_{1}^{o}  $ with  $ \lambda >\lambda_{m} 
$, \tred{the broken solution $u$ is not only the only stable solution, but also the global minimizer of $V_\ep [\lambda,\cdot]$. } \end{thm}
\bpr This follows by process of elimination: According to 
Proposition 4.6, the trivial solution  $ u\equiv 0 $ is unstable for all 
 $ \lambda \in (\lambda_{1} ,\infty ), $ and Proposition 4.7 implies that all 
higher-mode solutions are unstable. Moreover, the proof of Theorem 4.3 
implies that there are no monotone, strictly positive solutions  $ H>0 $  on 
 $ [0,\lambda ] $  for  $ \lambda >\lambda_{m} , $ where  $ \lambda_{m} <\infty  $  is 
some positive constant. Hence for fixed  $ \lambda >\lambda_{m} , $ the only 
potentially stable solutions are given by \eqref{eq45} and its anti-symmetric 
version generated by $ u(s)\to u(1-s)$.  Indeed, the phase-plane construction 
\eqref{eq44} demonstrates that there are no other possibilities for solutions with 
a single break. We then infer the result from Proposition \ref{prop2.1}.    \epr   
\begin{rem}After the bar breaks at $\la_\ast$, we can continue pulling the broken end $y=\la_\ast$ further to any $\la>\la_\ast$, cf. \eqref{eq44}. The interval 
$\la_\ast<y<\la$ is ``aether'' or vacuum $H=0$ or ``$F=\infty$'' namely a displacement discontinuity, known as the crack-opening displacement. See the discussion pertaining to \eqref{mb}. Here, broken solutions involve a two-phase inverse deformation, with the opened crack, $\la_\ast<y<\la$ in the broken phase, or the energy well at $H=0$ in Fig.~\ref{Fig2b}. The rest of the deformed bar  is in the unbroken phase (convex well containing $H=1$) with a transition layer in between, whose size depends on $\ep$. Despite this being a ``diffuse interface model'' due to higher gradients, the crack faces (boundaries of the $H=0$ phase) are sharply delineated, in contrast to the diffuse cracks in damage or phase field models \cite{bourdin}. This is due to the unilateral constraint. \end{rem}
\begin{rem} Theorems \ref{thm5.1} and \ref{thm5.2} imply that the bar breaks at a finite macroscopic stretch,  at most $\la_m$, beyond which the stress (of the only stable solution) vanishes. This is somewhat different than the behavior predicted by various nonlocal or cohesive zone models \cite{triant,dPTrusk}, where, for fracture,  the stress appears to approach zero only as the stretch goes to infinity.
Our result agrees with the discrete model \cite{bdg}, where however it is assumed that the constituent springs break (the force vanishes) at finite stretch. In contrast, our the underlying stress-stretch law is not restricted to vanish at finite homogeneous stretch. Here stress vanishes at finite average stretch due to bifurcation. \end{rem}

In order to characterize the projection $\mathcal{T}$ \eqref{slt} more explicitly, we turn to the methodology of Carr, Gurtin \& Slemrod\cite{cgs}. For
$0\le a <b<1$ and $z\ge 0$, let
\begin{equation}
\label{heq2}
 U(z,a,b)= W^{\ast }(z)-\varpi(a,b) z+\Gamma(a,b),\end{equation}
where
\beq\label{pg} \varpi(a,b)= {\frac{W^{\ast }(b)-W^{\ast }(a)}{b-a}}, \qquad \Gamma(a,b)=\varpi(a,b) a-W^{\ast }(a),
\eeq
and (formally) define
\beq\label{gg}g_0(a,b)=\int_a^b\frac{1}{\sqrt{U(z,a,b)}}dz, \quad g_1(a,b)=\int_a^b\frac{z}{\sqrt{U(z,a,b)}}dz
\eeq
\begin{prop}
\label{prop5.3}(i) Suppose $(\la,u)\in \overline {\mathcal{C}_{1}^{o} }$ (the first bifurcating sub-branch up to fracture) with $H(y)$ as in \eqref{eq3}. Let
$H(0)=H_2$ and $H(\la)=H_1$. Then these satisfy
\beq\label{cond1}0\le H_1<\kappa, \quad H_1<H_2<1,\eeq
\beq\label{cond1.1}U(H,H_1,H_2)>0 \quad \forall H\in(H_1,H_2),\eeq
\begin{equation}
\label{cond2}\sqrt{\frac{\ep}{2}}g_0(H_1,H_2)=\la,\qquad \sqrt{\frac{\ep}{2}}g_1(H_1,H_2)=1.
\end{equation}
Moreover the stress $\sigma$ from \eqref{eq58} is given by
\beq\label {ss}\sigma=W^{\ast }(H_1)-\varpi(H_1,H_2) H_1 \eeq
Conversely, suppose $H_1$, $H_2$ abide by \eqref{cond1},  \eqref{cond1.1} and satisfy \eqref{cond2}$_2$. Define $\lambda$ by \eqref{cond2}$_1$. Then there is $H:[0,\lambda]\to[0,1]$ with $H(0)=H_2$ and $H(\la)=H_1$, such that the corresponding $(\la,u)\in \overline {\mathcal{C}_{1}^{o} }$. The inverse of  this $H$ is given by
\beq\label{yh}\hat y(H)=\sqrt{\frac{\ep}{2}}\int_{H}^{H_2}\frac{1}{\sqrt{U(z,H_1,H_2)}}dz\eeq
for $H\in (H_1,H_2)$. The projection \eqref{slt} of the corresponding $(\la,u)$ onto the stress-stretch plane is $(\la,\sigma)\in\TT$ with $\la$ given by \eqref{cond2}$_1$ and $\sigma$ given by \eqref{ss}.
\hfill\break 
(ii) Setting $H_1=0$ in part (i) above gives the fracture point $(\la_\ast,u_\ast)$. In particular, the fracture stretch
$$\la_\ast=\sqrt{\frac{\ep}{2}}g_0(0,H_2),\; \hbox{ where $H_2$ is a root of } \sqrt{\frac{\ep}{2}}g_1(0,H_2)=1,$$
whereas the corresponding stress $\si_\ast=0$.
\hfill\break 
%
%
(iii) The projection $\TT$ in \eqref{slt} is located to the right of the rising branch of the stress-stretch curve, $\{(\la,\Dot W(\la)): 0<\la<1/\kappa\}$, namely,  
\beq\label{TTb}(\la,\si)\in\TT\implies 0\le\sigma<\dot W(1/\kappa),\quad \la>\la_\si,\eeq
where  $\la_\si$ is the unique solution of $\dot W(\la_\si)=\si$ in $[1,1/\kappa]$.
Moreover, as $\ep\to 0$,  $\TT$  approaches the rising branch of the stress-stretch curve 
 in the following sense. For fixed $\si$ with $(\la,\si)\in\mathcal{T}$, 
\beq\label{brl} \la\searrow \la_\si \quad\hbox{as }\ep\searrow 0. \eeq
 In particular, the fracture stretch $\la_\ast\searrow 1$ as $\ep\searrow 0$.
\end{prop}
%
%
\bpr 
By hypothesis,  Proposition \ref{prop4.2} with $n=1$ and the phase portrait (see discussion leading to \eqref{eq40}), $H$ is strictly monotone on $[0,\la]$ and $H'(y) \not=0$ except at $y=0$, $\la$. Thus \eqref{eq39} and the natural boundary conditions $H'(0)=H'(\la)=0$ from \eqref{eq12}$_2$, imply that  $W^{\ast }(H)-\varpi H+\Gamma=0$  for $H=H(\la)=:H_1$ and $H=H(0)=:H_2$,  and $>0$ for $H\in(H_1,H_2)$.  This in turn shows that 
\beq\label{u}W^{\ast }(H)-\varpi H+\Gamma=U(H,H_1,H_2)\eeq
cf. \eqref{heq2}, \eqref{pg}, and that  \eqref{cond1.1} holds, whereas $U(H_i,H_1,H_2)=0$ for $i=1,2$. This means that the straight line 
$\varpi H-\Gamma$ intersects  the graph of  $ {{W}}^{\ast }(\cdot )$ at $H_1$ and $H_2$, and is below it in between. Because of \eqref{eq-2},  this is only possible if \eqref{cond1} holds.
Substituting \eqref{u} in  \eqref{eq39}, solving for $H'$ and keeping the negative of the two solutions (the other giving equivalent results)  we infer
\beq
\label{heq}
H' (y)=-\sqrt{(2/\ep) \,U(H(y),H_1,H_2)},\quad y\in[0,\la]
\eeq
This can be solved for the inverse  $y=\hat y(H)$ of $H(y)$,  noting that $\hat y(H_2)=0$, yielding \eqref{yh}.
Using the latter, the requirement that $\hat y(H_1)=\la$ then gives the first of \eqref{cond2}, while
the integral constraint \eqref{eq2} reduces to the second of \eqref{cond2} after changing variables from $y$ to $H$.
 Also, \eqref{ss} follows from \eqref{gs} and \eqref{pg}. \par
 To show the converse,  suppose $H(y)$ is  a solution of  \eqref{heq} with $H(0)=H_2$; by \eqref{cond1.1}   it is monotone. Define $\hat y$ from \eqref{yh}. Then it is the inverse of $H(y)$, and  \eqref{cond2}$_1$ implies that $\hat y(H_1)=\la$.  Hence $H(0)=H_2$, $H(\la)=H_1$ and \eqref{cond2}$_2$ ensures the integral constraint  \eqref{eq2}, while  \eqref{heq} implies \eqref{eq39}. As a result the corresponding $(\la,u)\in\overline  {\mathcal{C}_{1}^{o} }$. \par
Part (ii) is immediate, after setting $H_1=0$ in \eqref{cond2} and \eqref{ss}, recalling that $\sw(0)=0$.  \par
To show (iii), we note that the mapping $(H_1,H_2)\mapsto (\varpi,\Gamma)$ defined by \eqref{pg} is one-to-one on the set of $(H_1,H_2)$ satisfying \eqref{cond1} and \eqref{cond1.1} \cite{cgs}.  We then rewrite 
\beq\label{pr1}g_i(H_1,H_2)=G_i(\varpi,\Gamma)=G_i(\varpi,-\si)\eeq
 for $i=0,1$,  in view of \eqref{gs}. For fixed $\si\in  [0,\dot W (1/\kappa))$ there is an interval  of $\vp$ values for which the chord $\varpi H+\si$ intersects  the graph of  $ \sw(\cdot )$ at two points $H_1<H_2$, and is below it in-between. This happens if and only if $\vp_\si<\vp<\vp^\si$ where  $\vp^\si:=\dot\sw(H^\si)>\vp_\si:=\dot\sw(H_\si)$ are slopes of rays  through the point  $(0,\si)$ and tangent to the graph of $\sw$ at $H^\si$ and $H_\si$, respectively. As a result, using \eqref{shield} and results from \cite{shield},
\beq\label{pr2}1/H_\si=\la_\si:= \hbox{the unique solution of $\dot W(\la_\si)=\si$ in $[1,1/\kappa]$}\eeq
Because of \eqref{eq57}, we have $\la>1/H_2$. From \eqref{eq-2} it follows that $H_\si>H_2$. As a result, $\la>\la_\si$. Also, if $\si\ge \dot W(1/\kappa)$, there can be at most one intersection, so the only possible  $(\la,\si)\in\TT$ is the bifurcation point $(\la_1,\dot W(\la_1)$. But since $\la_1>1/\kappa$, $\si_1=\dot W(\la_1)<\dot W(1/\kappa)$, so if  if $\si\ge \dot W(1/\kappa)$, then $(\la,\si)\not\in\TT$.  This confirms \eqref{TTb}.
Adapting results from \cite{cgs} we have that for some constant $C>0$,
\beq\label{pr3}G_0(\vp,-\si) \sim -C\log(\vp-\vp_\si),\quad G_1(\vp,-\si) \sim -H_\si C\log(\vp-\vp_\si), \quad\hbox{as } \vp\searrow\vp_\si
\eeq
Note that \eqref{cond2}$_2$ becomes 
$$G_1(\vp,-\si)=\sqrt{2/\ep}.$$
For $\ep$ sufficiently small  this has a solution \cite{cgs}
\beq\label{pr4}\vp=\vp_\ep \sim \vp_\si+e^{-C\sqrt{2/\ep} } \hbox{ as } \ep\to 0.\eeq 
Dividing \eqref{cond2}$_1$ by \eqref{cond2}$_2$ and using \eqref{pr1}-\eqref{pr3}, we find
$$\la=\frac{G_0(\vp_\ep,-\si)}{G_1(\vp_\ep,-\si)}\to \la_\si \; \hbox{as } \ep\to0.$$
Together with \eqref{TTb} this confirms \eqref{brl}.
\epr
\begin{rem} In the context of the original formulation, the analogues of \eqref{heq2}-\eqref{cond2} are due to \cite{cgs}; see also \cite{triant}.  For each $H_1\in[0,1/\la_1)$, the second of \eqref{cond2} is uniquely solvable for $H_2=\hat H(H_1)$.
Define
\beq
\label{bran} \hat\la (H_1,\ep)=\sqrt{\frac{\ep}{2}}g_0(H_1,\hat H(H_1)), \quad \hat \sigma (H_1)=W^{\ast }(H_1)-\varpi(H_1,\hat H(H_1)) H_1
\eeq
Then the above is a parametrization  of $\TT$, cf. \eqref{slt}, in the $(\la,\sigma)$ plane. In specific examples we (numerically) solve the second of \eqref{cond2} for $H_2$ in terms of $H_1$. This gives the parametrization \eqref{bran} of $\TT$.
\end{rem}
Next we consider some properties of stable broken solutions (of the Euler-Lagrange inequality corresponding to \eqref{eq1}) on the first branch $\mathcal{C}_{1}$ with the inverse stretch vanishing on an entire  subinterval. Such solutions have $H=H_b(y)$ given by the right-hand side of \eqref{eq44} for $y\in[0,\la]$ and $\la>\la_\ast$ as guaranteed by Theorems \ref{thm4.5} and \ref{thm5.1}. 
\begin{prop}
\label{prop5.4}
As $\ep\to 0$, the energy  \eqref{eq1} of a broken solution $H_b$ is
\beq\label{ener} E_\ep [H_b]=\sqrt{\ep}\int_0^1 \sqrt{2W^\ast (H)} dH +o(\sqrt{\ep}) = \sqrt{\ep}\int_1^\infty \sqrt{2W (F)/F^5} dF +o(\sqrt{\ep})\eeq
\end{prop}
\bpr
Modifying an argument of \cite{cgs}, we write the energy \eqref{eq1} of $H_b$ as follows, observing the constraint $\int_0^\la H_b(y)dy=1$, noting that $H_b=0$ on $[\la_*,\la]$ by \eqref{eq44} and using \eqref{heq2},
\begin{align*}
E_\ep [H_b] &=  \int_0^\la \Bigl(\frac{\ep}{2}[H'_b(y)]^2+W^*(H_b(y))-\varpi[H_b(y)-1/\la]\Bigr)dy \\
&= \int_0^{\la_*} \Bigl(\frac{\ep}{2}[H'_b(y)]^2+W^*(H_* (y))-\varpi H_* (y)\Bigr)dy+\varpi \\
&= \int_0^{\la_*} \Bigl(\frac{\ep}{2}[H'_b(y)]^2+U(H_* (y),0,H_2(\ep)\Bigr)dy+\varpi(0,H_2(\ep)) \\
&= \int_0^{\la_*} 2U(H_* (y),0,H_2(\ep))dy+\varpi(0,H_2(\ep)) \\
&=\sqrt{\ep} \int_0^{H_2(\ep)} \sqrt{2U(H,0,H_2(\ep))}dH+\varpi(0,H_2(\ep)) \\
\end{align*}
where we have used \eqref{heq} to obtain the fourth line above and the change of variables \eqref{yh} from $y$ to $H=H_b(y)$ for the fifth line. Here $H_2(\ep)$ is the root of   the second of \eqref{cond2} with $H_1=0$. As $\ep \to 0$,\eqref{pr4} applies, whereas $\vp_\si=0$ because $\si=0$, so that the second term above $\varpi(0,H_2(\ep))=o(\sqrt{\ep})$.  From \eqref{pr4} and the fact that $\sw(H)=O(H-1)^2$ as $H\to1$, we have that $H_2(\ep)=1+o(\sqrt{\ep})$. As a result we can replace the upper limit in the  integral in the last line above by $1$.
\epr

\begin{rem} Here there is a transition layer from $H$ close to $1$ to $H=0$ of size approximately $\sqrt{\ep}$ for small $\ep$. In \eqref{eq44}, this layer occurs just to the left of the crack face $y=\la_\ast$. This is easily shown.  Moreover
the first formula in \eqref{ener} is formally identical to the interfacial energy of a phase boundary with higher gradients \cite{cgs}, so this energy can be interpreted as the energy cost for the creation of new surfaces, or the surface energy of fracture in the sense of Griffith \cite{griffith}.
\end{rem}
In order to gain more information concerning the stability of the first branch, the shape of its projection $\TT$ in the $(\la,\si)$ plane and the location of the fracture stretch 
$\lambda_{\ast } 
$,  we turn to  specific models for $\sw$ and 
compute the first global branch of solutions.

\begin{figure}
  \centering
\includegraphics[width=0.8\textwidth]{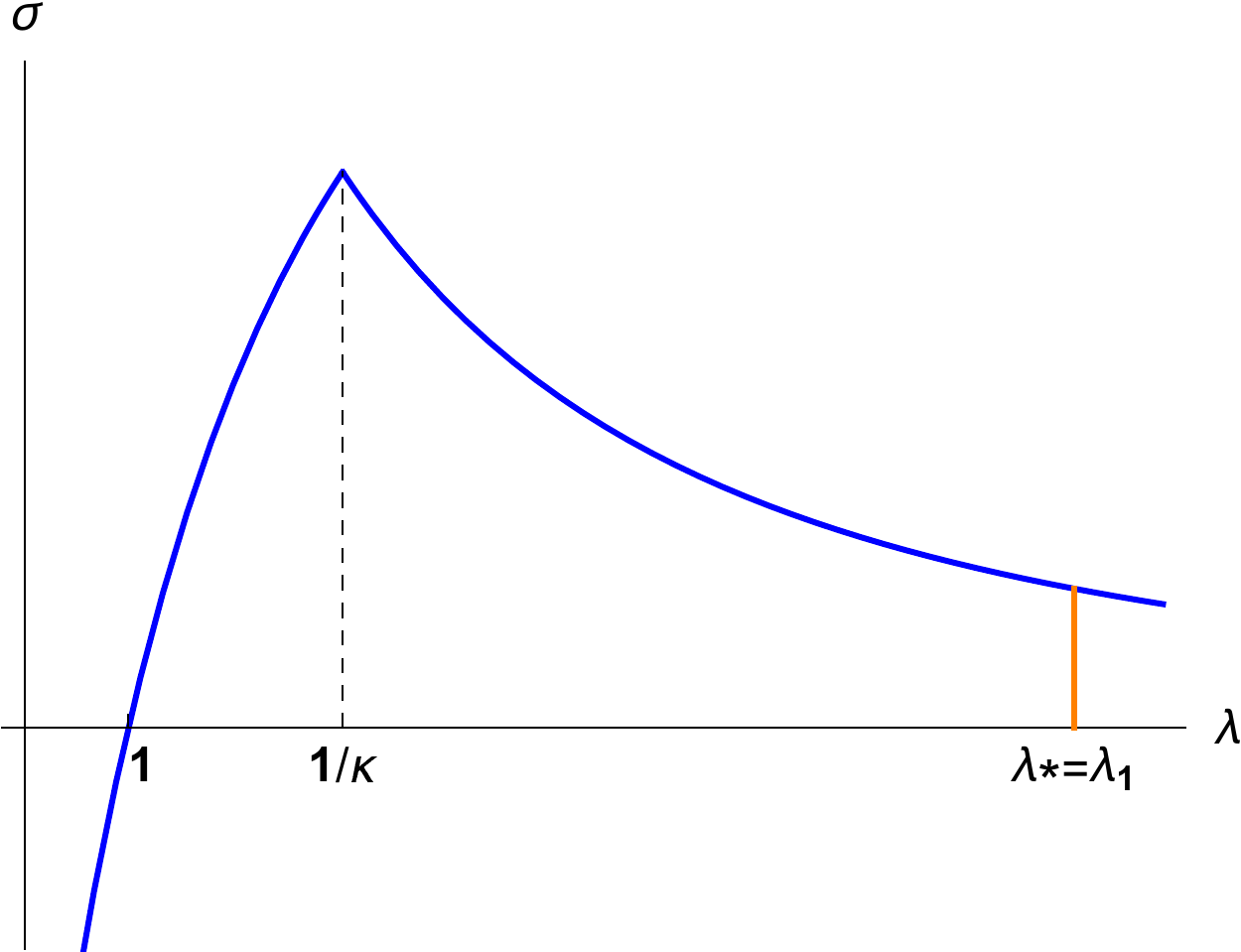}
  \caption{The homogeneous stress \textit{vs} stretch trivial branch (blue) and the projection $\mathcal{T}$ (orange) of the first bifurcation branch for the special constitutive law \eqref{W} in 
  Example \ref{ex5.9} for 
   $ \ep  =  16/\pi^{2} $.}
  \label{Fig7}
\end{figure}

\begin{ex}\label{ex5.9}
 \rm We introduce a special $W^\ast$ that is piecewise-quadratic, so that the Euler-Lagrange equation is (piecewise) linear. Let
$$\kappa=1/\sqrt{2},\quad d=\sqrt{2}-1$$
and define
\begin{equation}\label{W} W^\ast(H)=\begin{cases}
d^2-(H-d)^2, & 0\le H\le \kappa, \\
(H-1)^2, & \kappa<H<\infty.\\
\end{cases}
\end{equation}
\begin{figure}
  \centering
 \subfloat[]{\includegraphics[width=0.42\textwidth]{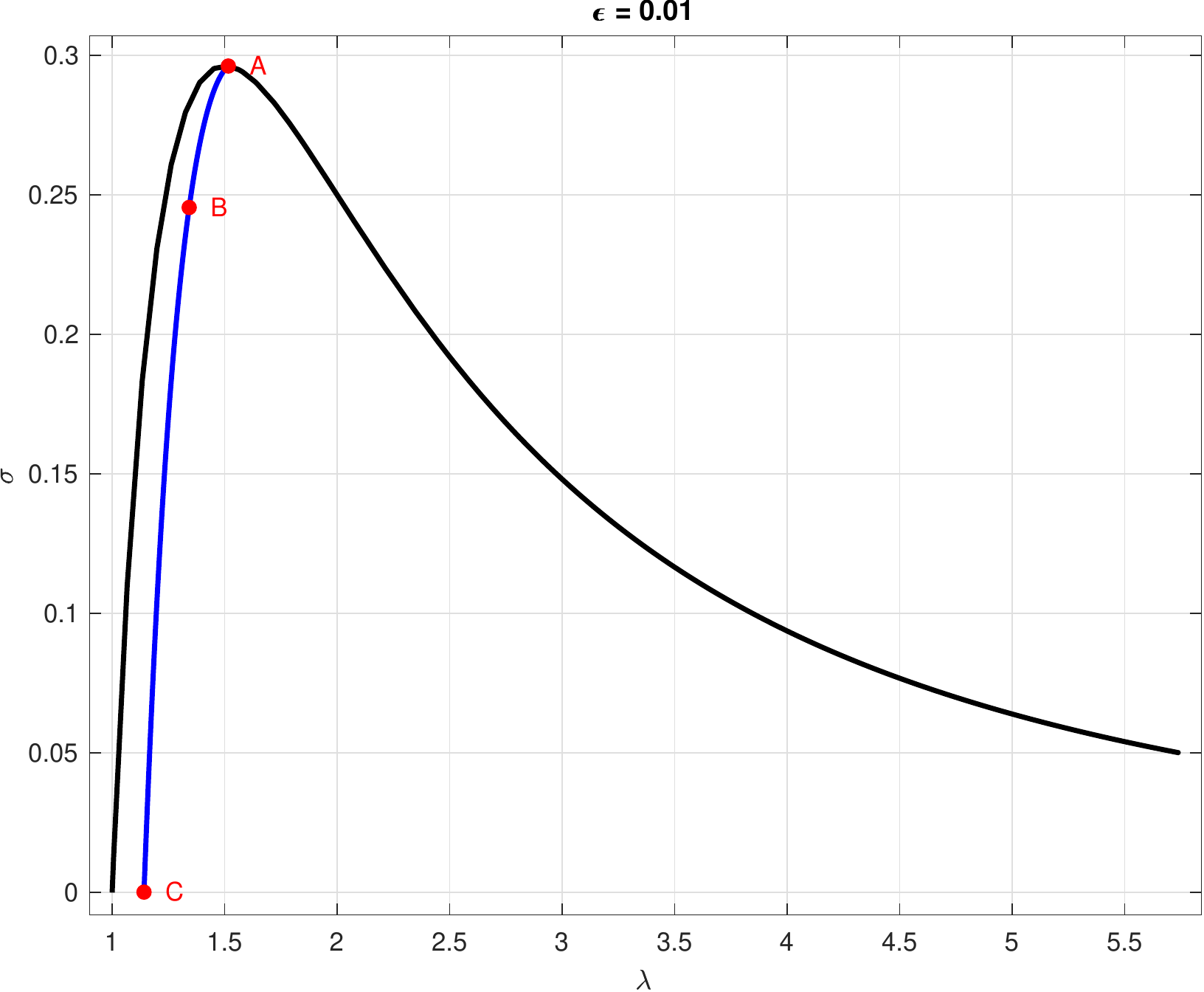}
  \label{Fig8a}}
  \hspace{0.8cm}
 \subfloat[]{\includegraphics[width=0.42\textwidth]{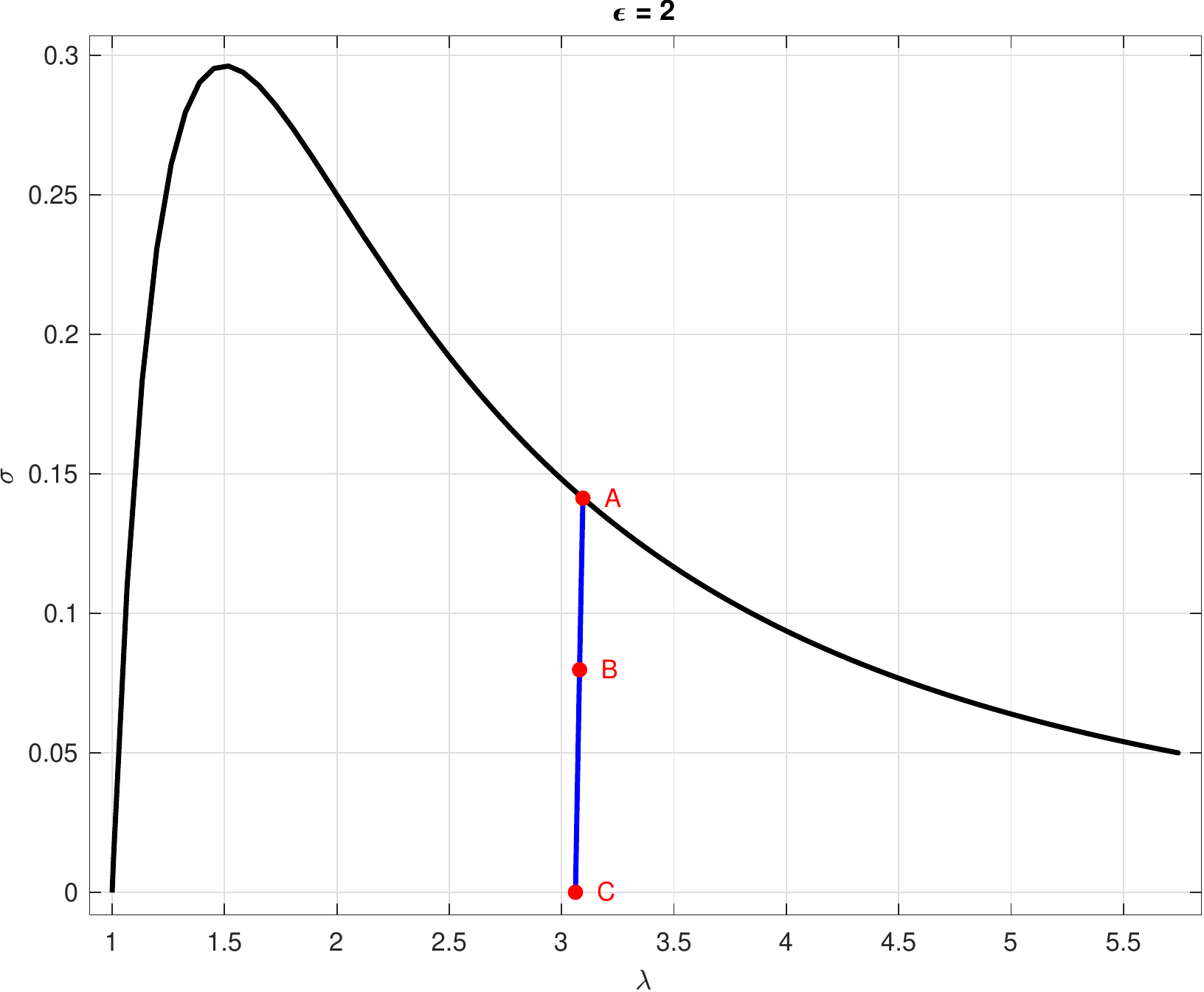}
  \label{Fig8b}} 
  \caption{The homogeneous stress \textit{vs} stretch trivial branch (black) and the projection $\mathcal{T}$ (blue) of the first bifurcation branch for the special constitutive law \eqref{WW} in Example \ref{ex5.10}. (A)  $ \ep  =  0.01$; (B) $ \ep  =  2$. Point $A$ is the bifurcation point $(\la_1,\dot W(\la_1)$. Point $C$ is the fracture point $(\la_*,0)$.}
  \label{Fig8}
\end{figure}
\begin{figure}
    \centering
      \subfloat[]{\includegraphics[width=0.4\textwidth]{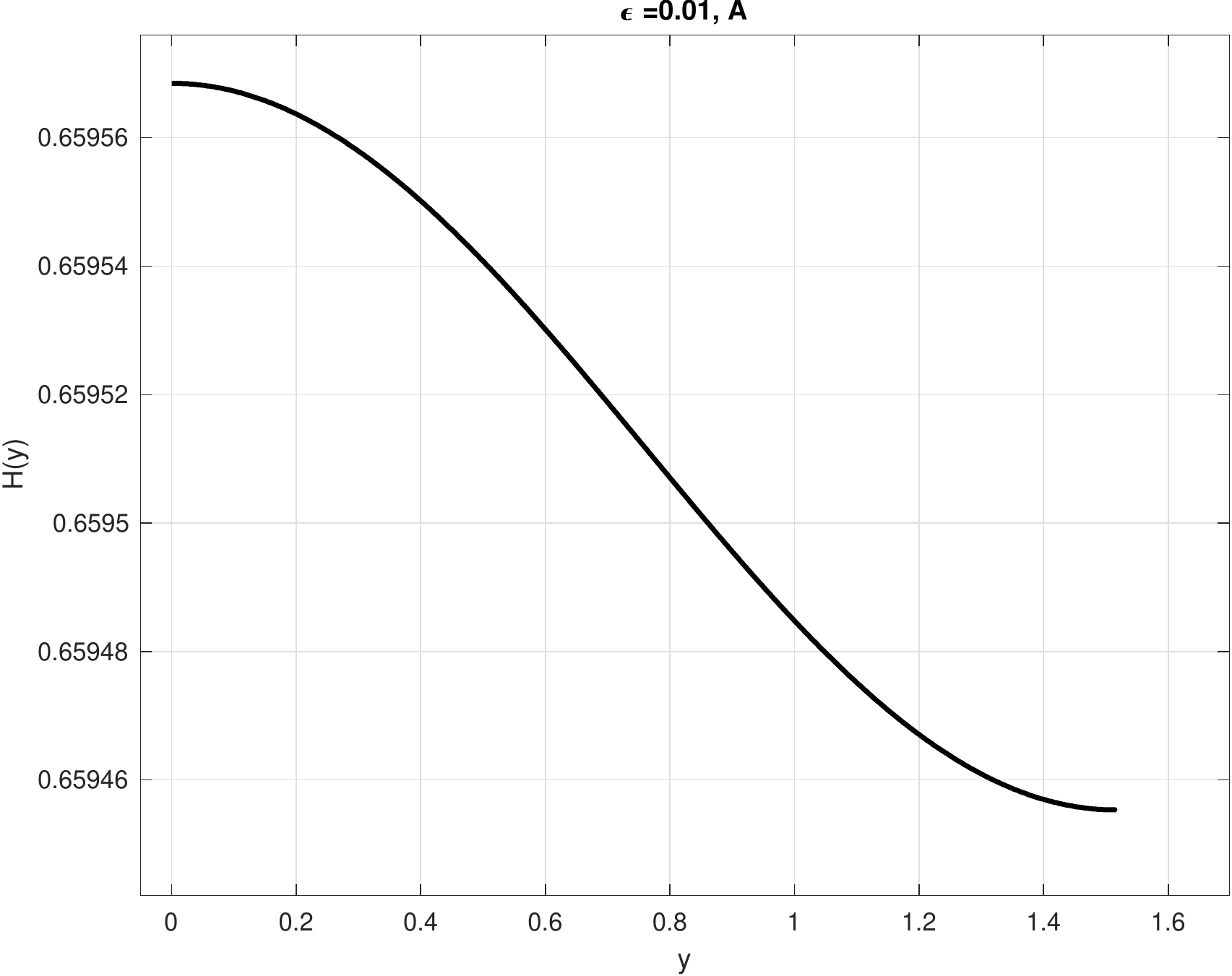}    \includegraphics[width=0.4\textwidth]{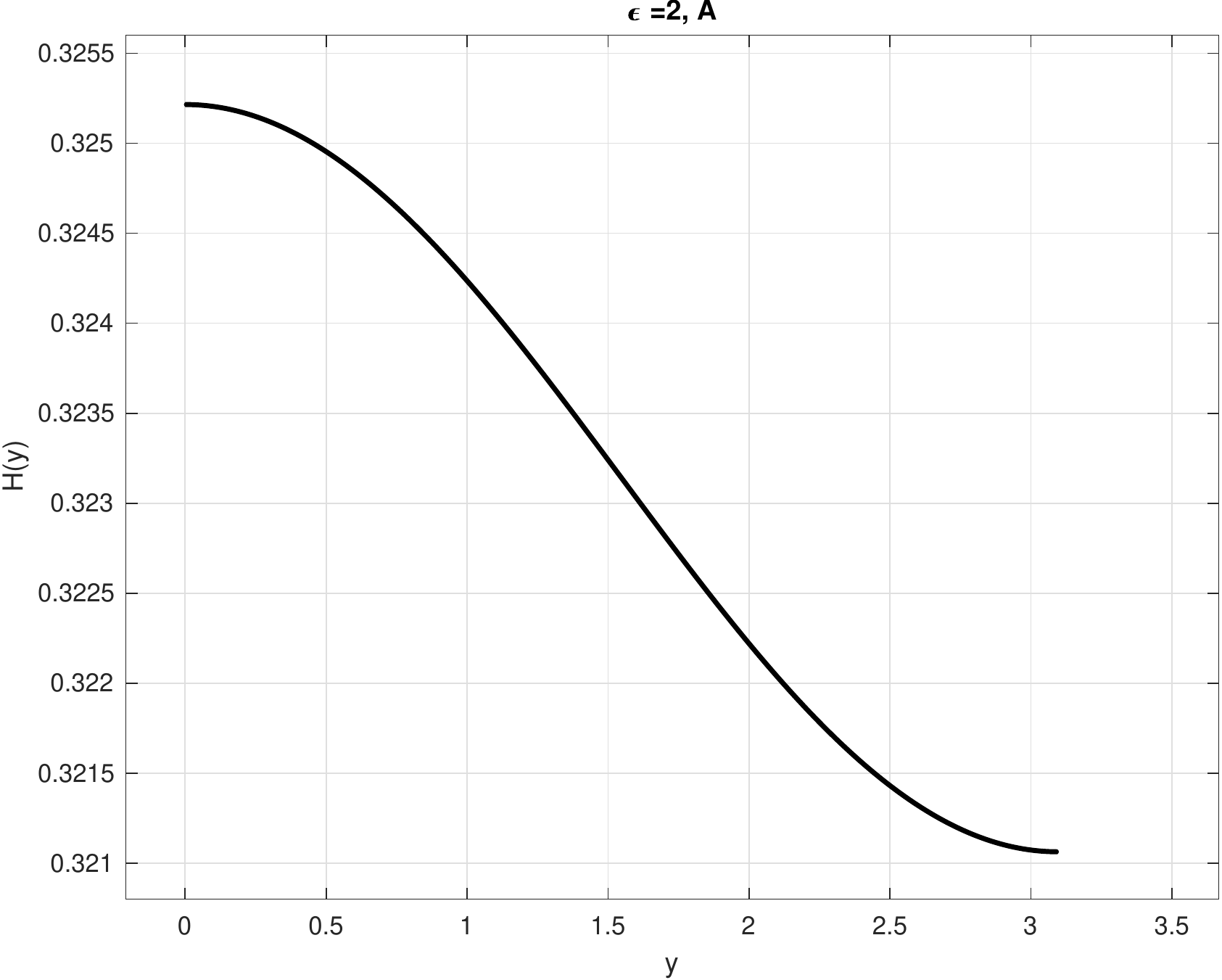}
      \label{Fig9a}}
   \vskip.5cm
    \subfloat[]{\includegraphics[width=0.4\textwidth]{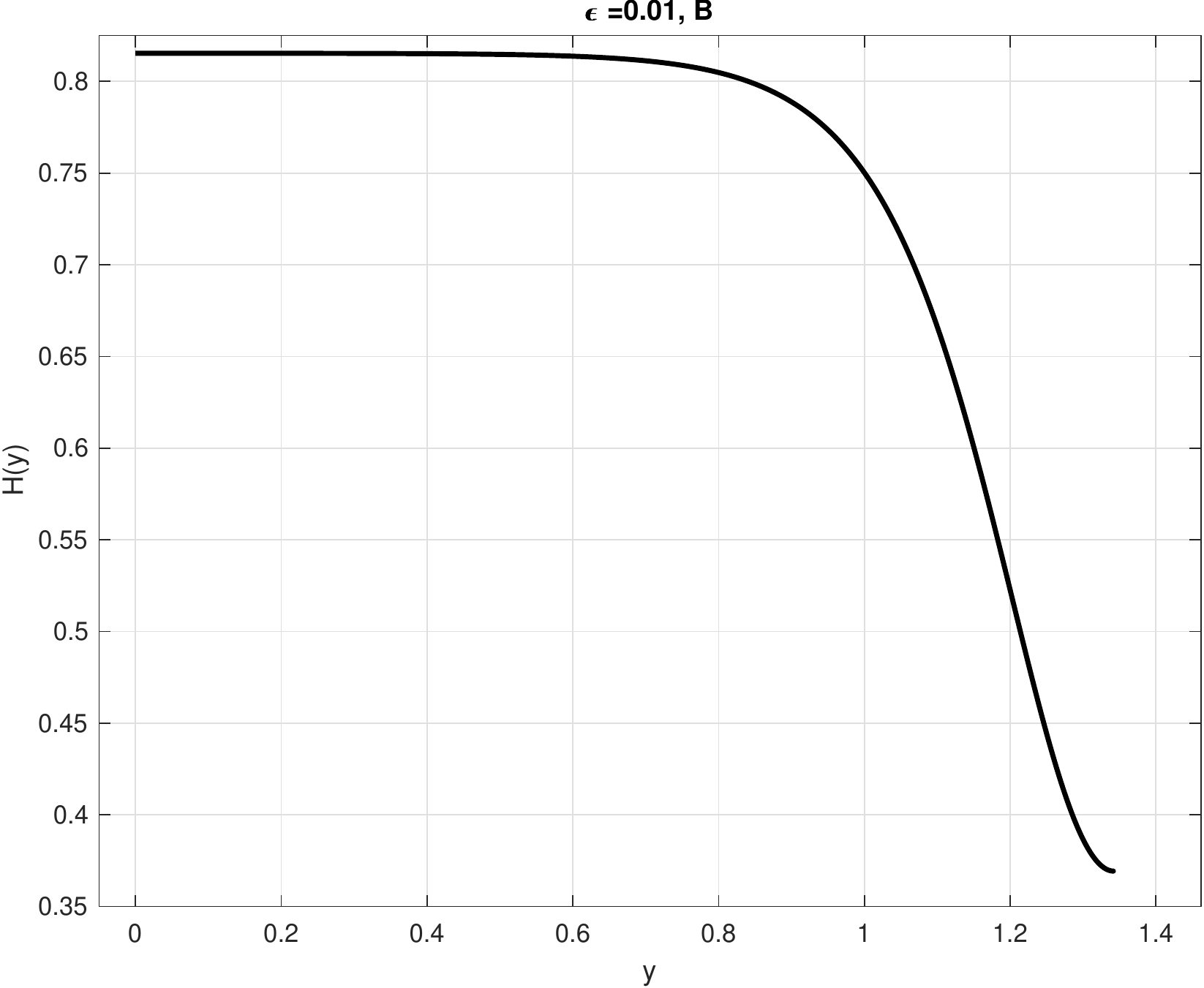}
        \includegraphics[width=0.4\textwidth]{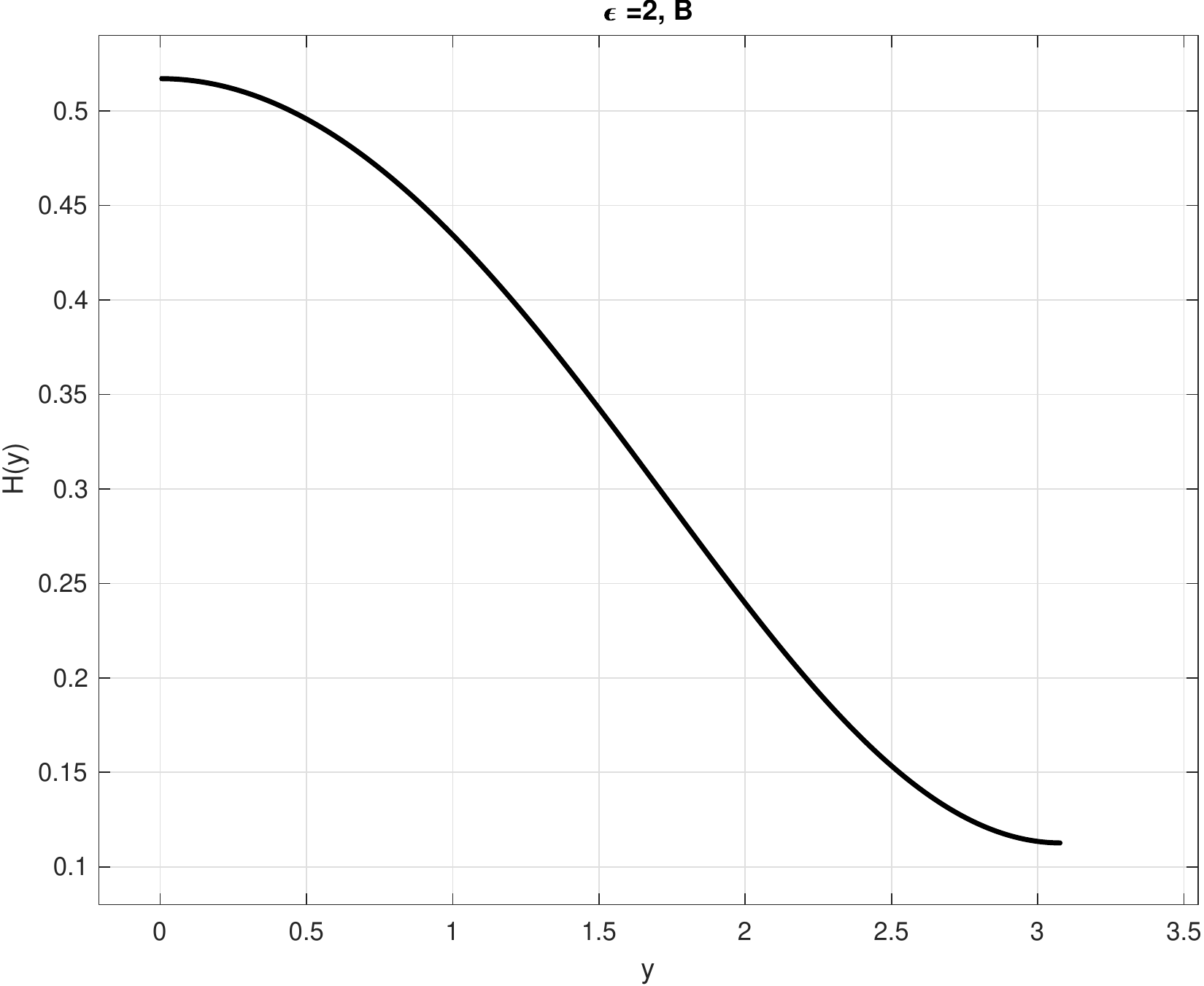}
      \label{Fig9b}}  
        \vskip.5cm
    \subfloat[]{\includegraphics[width=0.4\textwidth]{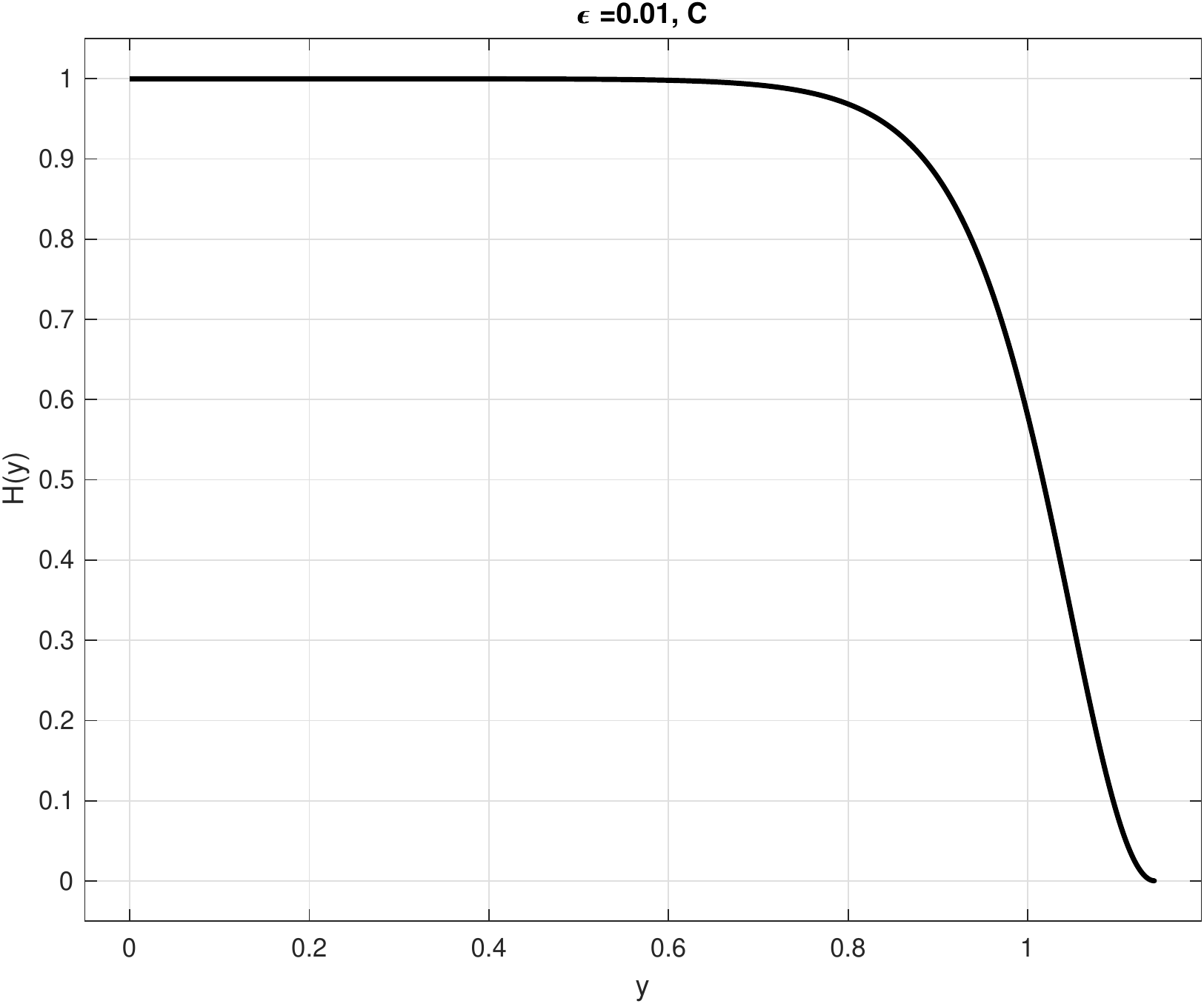}
        \includegraphics[width=0.4\textwidth]{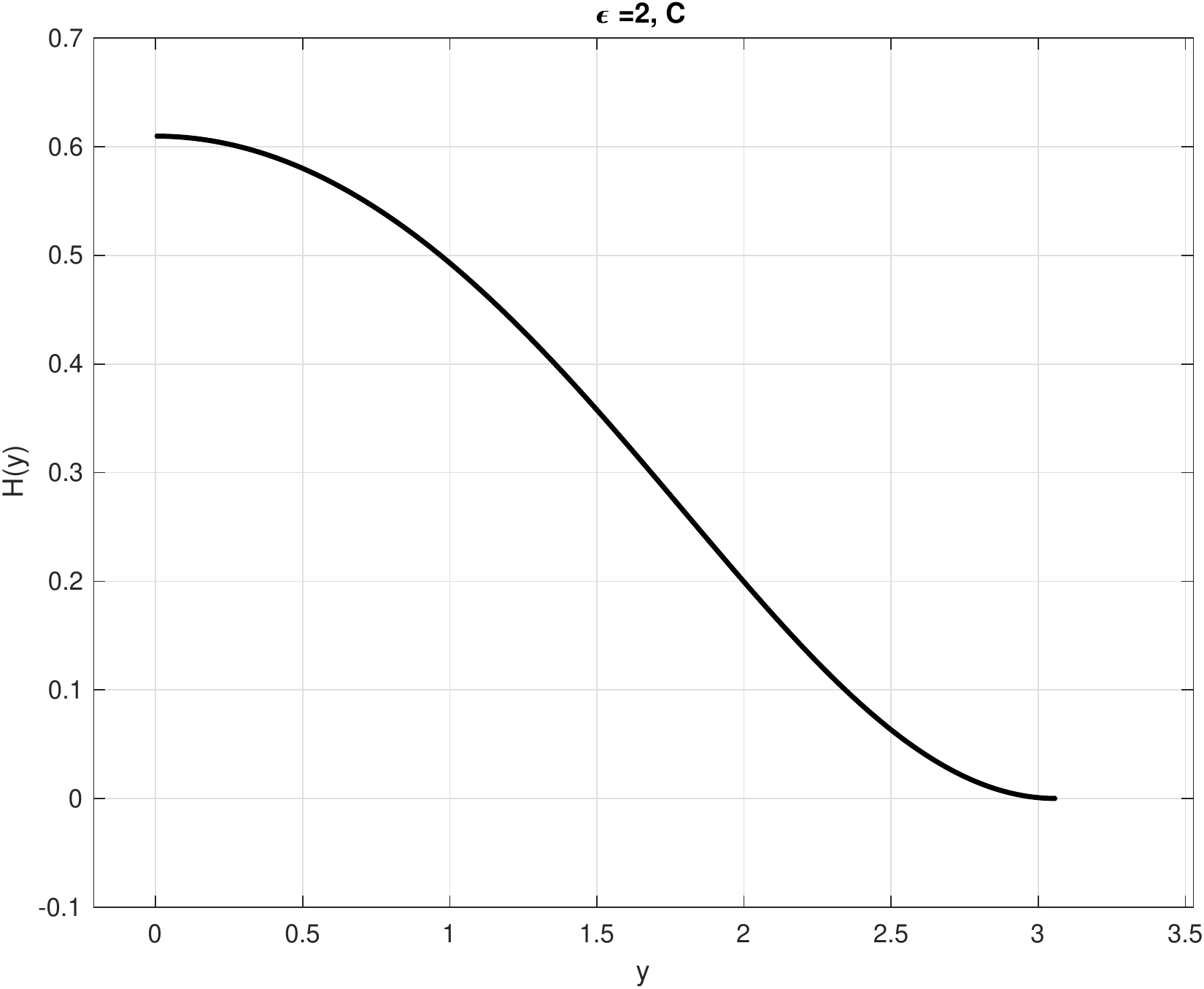}
      \label{Fig9c}}
\caption{Bifurcated first-branch solutions  $H(y)$ \textit{vs} $y$ from points A, B, C along the first-branch projection in Fig.~\ref{Fig8}  for the special constitutive law \eqref{WW} in Example \ref{ex5.10}.  Left column:  $ \ep  =  0.01$; right column: $ \ep  =  2$. (A) near the bifurcation point; note small variation of $H$. (B) At an intermediate point along the sub-branch. (C) at the fracture point.}
 \centering
 \label{Fig9}
\end{figure}
This abides by \eqref{eq-2} except for smoothness, as it is $C^1$ but only piecewise $C^2$, but this does not affect our conclusions here.
Suppose we look at solutions of \eqref{eq12} such that $0<H(x)<\kappa$ for $0<x<\lambda$, so that 
\begin{equation}\label{ineq} -1<u(s)<\kappa\lambda-1, \quad s\in(0,1)\end{equation}
 in  \eqref{eq12} (this demands that $\lambda>1/\kappa$). Then  \eqref{eq12} becomes linear and reduces exactly to \eqref{eq27}, whose solutions are 
$u(s)=A\cos(n\pi s)$, $0\le s\le 1$, for some constant $A$, with $\lambda$ restricted to satisfy \eqref{eq29}, namely $\lambda=\lambda_n= n\pi{\sqrt{\varepsilon/2}}$. The bifurcation condition thus holds all along  each branch, from $A=0$ at bifurcation, all the way to  fracture where $|A|=1$, so that $u=-1$ at one end. The second inequality in \eqref{ineq}  then asserts $\kappa\lambda\ge2$, which is equivalent to 
 $ \varepsilon  \ge  16/\pi^{2} $  for  $ n=1$.  The projection of each of the 
branches on the  $ (\lambda ,\sigma ) $  plane is a vertical line from the 
bifurcation point  $ (\lambda_{n} ,\Dot{{W}}^{\ast }(\lambda_{n} )) $  on the 
trivial branch, all the way down to $ (\lambda ,\sigma )=(\lambda_{n} ,0)$.  
In particular, note that  $ \lambda_{\ast } =\lambda_{1} =\pi \sqrt 
{\varepsilon /2} $ and the projection $\mathcal{T}$  of the first bifurcation branch is the orange vertical line in Fig.~\ref{Fig7}..  According to the construction \eqref{eq44}, \eqref{eq45}, the broken 
part of the branch for all  $ \lambda >\lambda_{1}  $  corresponds to
\begin{equation}
\label{eq62}
u(s)=\begin{cases}
 (\lambda/\lambda_1 )\cos(\lambda \pi s/\lambda_1), & s\in [0,\lambda_{1} /\lambda ),  \\
 -1,&  s\in [\lambda_1 /\lambda ,1]. 
\end{cases} 
\end{equation}
The bifurcation condition \eqref{eq29} holds along the entire branch 
 $ \mathcal{C}_{1}^{o}  $  in this case, which implies that the second variation 
\eqref{eq47} can vanish. To see this, note that integration by parts in \eqref{eq47} 
yields
\begin{equation}
\label{eq63}
\delta^{2}V_{\varepsilon } [\lambda ,u;\eta ]=\int_0^1 {[-\varepsilon {\eta 
}''} +\lambda^{2}\Ddot{{W}}^{\ast }([1+u]/\lambda )\eta ]\eta ds,
\end{equation}
for all  $ \eta \in \mathcal{H}$.  Here  $ \Ddot{{W}}^{\ast }\equiv -2, $  and the 
integrand vanishes at  $ \lambda =\lambda_{1}  $  and  $ \eta (s)=\cos (\pi s), $  
as already noted. In fact, we claim that the second variation is positive 
semi-definite: If we restrict variations to the orthogonal complement of 
 $ \mbox{span}\{\cos (\pi s)\} $ in $ \mathcal{H}, $ denoted 
 $ \Tilde{{\mathcal{H}}}, $ then the sharp Poincar\'{e} inequality \eqref{eq49} now 
incorporates the second eigenvalue of the operator, viz.,  $ 4\pi^{2}, $ and 
for  $ 0<\left| A \right| \le  1, $ we obtain 
\begin{equation}
\label{eq64}
\delta^{2}V_{\varepsilon } [\lambda_{1} ,A\cos (\pi s);\eta ] \ge  
3\varepsilon \pi^{2},
\end{equation}
for all  $ \eta \in \Tilde{{\mathcal{H}}}$.  Hence, in this special case,  $ u $ is 
locally ``neutrally'' stable for all  $ (\lambda ,u)\in \overline 
{\mathcal{C}_{1}^{o} } \backslash \{(\lambda_{1} ,0)\}. $  
\end{ex}
\begin{ex}\label{ex5.10} \rm Next we consider  
\beq\label{WW}   W(F)=(1-1/F)^{2}\implies W^{\ast }(H)=H(1-H)^{2}, \eeq  
which follows from 
 from \eqref{shield} and also fulfills our hypotheses \eqref{eq-2}. See Fig.~\ref{Fig2} for $W$ and $\sw\!$. It is enough to 
solve \eqref{eq12} in order to compute the sub-branch  $ \mathcal{C}_{1}^{o}  $  and 
its projection onto the  $ (\lambda ,\sigma ) $  plane. For that purpose, we 
used the bifurcation/continuation software AUTO  \cite{auto}. In Fig.~\ref{Fig8a} and Fig.~\ref{Fig8b} we 
depict the computational results for the cases  $ \varepsilon =0.01 $ and 
 $ \varepsilon =2, $ respectively. The computed values of  $ \lambda_{\ast }  $  in 
these two cases are approximately $1.14$ and $3.06$, respectively. In each of 
the two figures we specify three points along the curve at which the 
corresponding computed solution configurations are depicted in Fig.~\ref{Fig9}. For $\ep=0.01$, observe the development of a transition layer from $H\approx1$ to $H\approx0$ on the right-hand side of the bar in Fig.~\ref{Fig9b} and Fig.~\ref{Fig9c}. Fracture occurs at the rightmost point $y=\la_\ast$ where $H(\la_\ast)=0$ in Fig.~\ref{Fig9c}. .
\par
We also use the semi-analytical solution of Proposition \ref{prop5.3}(i).  We numerically compute the set of $(H_1,H_2)$ satisfying \eqref{cond2}$_2$ and \eqref{cond1}.  We then compute the corresponding $(\la,\si)\in\TT$ from \eqref{cond2}$_2$ and \eqref{ss}.  The projection $\TT$ of the first branch so obtained is shown in Fig.~\ref{Fig4} for three values of $\ep$. As predicted by part (ii) of Proposition \ref{prop5.3}, as $\ep$ decreases, $\TT$ tends towards the rising branch of the stress-stretch curve.  
 \end{ex}
 
 \begin{rem} In both the above examples, a quasistatic increase of  average stretch $\la$ (equivalently the  end displacement  of the bar) will cause a sudden drop of the stress $\si$ to zero upon fracture. This will occur at some stretch level that  is equal to the bifurcation point $\la_1$ in Example \ref{ex5.9}, or  some value between $\la_ast$ and $\la_1$ for Example \ref{ex5.10},
depending on one's favorite notion of stability (local  \textit{vs}  global), and perhaps ambient disturbances or imperfections. Many nonlocal or cohesive zone models \cite{triant,dPTrusk}
instead predict a load drop to a positive stress, followed by a gradual asymptotic decrease to zero.  In contrast a discrete model exhibits sudden drop to zero stress for long bars only, after assuming that ``interatomic springs'' lose force (break) at finite stretch. We do not need to assume this for our constitutive law. For shorter bars our model predicts a larger fracture stretch and a smaller stress drop to zero at fracture. \end{rem}

\bibliographystyle{unsrt}

\bibliography{FractureBib}

\end{document}